\documentclass[oneside,reqno,english]{amsart}
\usepackage[T1]{fontenc}
\usepackage[utf8]{inputenc}
\usepackage{xcolor}
\usepackage{babel}
\usepackage{prettyref}
\usepackage{amstext}
\usepackage{amsthm}
\usepackage{amssymb}
\usepackage[pdfusetitle,
 bookmarks=true,bookmarksnumbered=false,bookmarksopen=false,
 breaklinks=false,pdfborder={0 0 0},pdfborderstyle={},backref=false,colorlinks=false]
 {hyperref}
\hypersetup{
 colorlinks=true,citecolor=blue,linkcolor=blue,linktocpage=true}

\makeatletter
\numberwithin{equation}{section}
\numberwithin{figure}{section}

\usepackage{prettyref}

\newrefformat{cor}{Corollary~\ref{#1}}
\newrefformat{subsec}{Section~\ref{#1}}
\newrefformat{lem}{Lemma~\ref{#1}}
\newrefformat{thm}{Theorem~\ref{#1}}
\newrefformat{sec}{Section~\ref{#1}}
\newrefformat{chap}{Chapter~\ref{#1}}
\newrefformat{prop}{Proposition~\ref{#1}}
\newrefformat{exa}{Example~\ref{#1}}
\newrefformat{tab}{Table~\ref{#1}}
\newrefformat{rem}{Remark~\ref{#1}}
\newrefformat{def}{Definition~\ref{#1}}
\newrefformat{fig}{Figure~\ref{#1}}
\newrefformat{claim}{Claim~\ref{#1}}
\newrefformat{assu}{Assumption~\ref{#1}}

\makeatother

\theoremstyle{plain}
\newtheorem{thm}{\protect\theoremname}[section]
\newtheorem{prop}[thm]{\protect\propositionname}
\newtheorem{cor}[thm]{\protect\corollaryname}
\theoremstyle{remark}
\newtheorem{rem}[thm]{\protect\remarkname}
\theoremstyle{plain}
\newtheorem{lem}[thm]{\protect\lemmaname}
\providecommand{\corollaryname}{Corollary}
\providecommand{\lemmaname}{Lemma}
\providecommand{\propositionname}{Proposition}
\providecommand{\remarkname}{Remark}
\providecommand{\theoremname}{Theorem}

\begin{document}
\title[Bernoulli cylinder frame operators]{Bernoulli cylinder frame operators: filtration, Haar structure, and
self-similarity}
\begin{abstract}
We study the finite-rank frame operators generated by cylinder indicator
functions for the Bernoulli Cantor measure $\mu_{p}$. In the symmetric
case $p=\frac{1}{2}$, the natural Haar differences diagonalize these
operators. For general $0<p<1$, we show that the weighted Haar basis
still yields a sparse tree-banded matrix form, although diagonalization
is lost. We also prove a filtration representation in terms of conditional
expectations and level-wise mass operators. This leads to a norm convergent
limit operator $K_{\infty}$, which is compact, positive, and self-adjoint.
Finally, we show that $K_{\infty}$ is characterized by a self-similar
operator identity induced by the first-level Cantor decomposition,
and we derive corresponding block and scalar resolvent renormalization
formulas.
\end{abstract}

\author{James Tian}
\address{Mathematical Reviews, 535 W. William St, Suite 210, Ann Arbor, MI
48103, USA}
\email{james.ftian@gmail.com}

\maketitle
\tableofcontents{}

\section{Introduction}\label{sec:1}

Let $\mu_{p}$ be the Bernoulli Cantor measure on the middle-third
Cantor set $C$, with weights $p$ and $1-p$ on the two first-level
branches. The associated cylinder sets form a nested binary tree of
measurable subsets of $C$, and hence a natural family of indicator
functions in $L^{2}\left(\mu_{p}\right)$. This family carries two
structures at once: a geometric one, coming from the tree of cylinders,
and a metric one, coming from the cylinder masses determined by $\mu_{p}$. 

Given a finite depth $m$, one may consider the cylinder indicators
\[
\left\{ 1_{C_{u}}:|u|\le m\right\} \subset L^{2}\left(\mu_{p}\right)
\]
and the associated finite-rank positive operator 
\[
K_{m}f=\sum_{|u|\le m}\left\langle 1_{C_{u}},f\right\rangle 1_{C_{u}}.
\]
Equivalently, $K_{m}$ is the frame operator of the finite cylinder
family, or the Gram operator attached to the intersection kernel 
\[
K\left(A,B\right)=\mu_{p}\left(A\cap B\right)
\]
on cylinders. Thus the problem may be read either as a question about
finite systems of indicator functions in $L^{2}\left(\mu_{p}\right)$
or as a question about the operator theory of a concrete positive
kernel adapted to the Bernoulli Cantor tree.

At the finite level, the first phenomenon is a weighted version of
Haar structure. In the symmetric case $p=\frac{1}{2}$, the equal
split at each vertex produces the usual Haar differences on the tree,
and $K_{m}$ is diagonal in the resulting orthogonal basis. For general
$0<p<1$, one still has an orthogonal family of weighted sibling differences,
but the operator is no longer diagonal. Instead, the matrix of $K_{m}\big|_{\mathcal{F}_{m}}$
in the weighted Haar basis vanishes unless the two indices are comparable
in the tree. The finite operator therefore retains a strong sparsity
pattern, but it is a tree sparsity rather than a diagonal one. This
separates the symmetric and nonsymmetric Bernoulli cases in a useful
way.

A second finite level description comes from the natural filtration
by cylinders. Let $E_{n}$ be conditional expectation onto the level-$n$
cylinder $\sigma$-algebra, and let $D_{n}$ be multiplication by
the level-$n$ cylinder mass function. Then 
\[
K_{m}=\sum^{m}_{n=0}D_{n}E_{n}.
\]
This identity places the operators $K_{m}$ in a form closer to martingale
and filtration methods. It also gives direct norm control, and from
it one obtains convergence in operator norm of the series 
\[
K_{\infty}=\sum^{\infty}_{n=0}D_{n}E_{n}.
\]
Thus the finite cylinder frame operators converge to a compact positive
self-adjoint operator on $L^{2}\left(\mu_{p}\right)$.

The limiting operator is not only an object obtained by summing the
finite level contributions. It also admits a self-similar description
coming from the first-level splitting of the Cantor set. The two branches
define isometries $U_{0},U_{2}:L^{2}\left(\mu_{p}\right)\to L^{2}\left(\mu_{p}\right)$,
and if $\phi=1_{C}$, then the limit operator satisfies 
\[
K_{\infty}=P_{\phi}+p\,U_{0}K_{\infty}U^{*}_{0}+\left(1-p\right)U_{2}K_{\infty}U^{*}_{2}.
\]
Moreover, this identity determines $K_{\infty}$ uniquely among bounded
operators. In this form the operator $K_{\infty}$ is characterized
by the same self-similar splitting that defines the measure $\mu_{p}$.
This yields a second description of the limit object, one which is
no longer tied to truncation in the tree.

It is worth mentioning the order in which these structures appear.
The paper does not begin with an abstract transfer operator or an
abstract fixed-point problem and then introduce a model realizing
it. The finite frame operators come first. The weighted Haar structure
and the filtration formula arise directly from the cylinder system,
and the self-similar operator equation appears only afterward. In
that sense the fixed-point identity is not imposed externally; it
is extracted from the finite-dimensional construction itself. This
passage from concrete finite level operators to an intrinsic description
of the infinite one is a main thread running through the paper.

The self-similar identity has several consequences. Its linear part
is a normal completely positive contraction on $B\left(L^{2}\left(\mu_{p}\right)\right)$,
which yields a norm convergent Neumann expansion for $K_{\infty}$.
The same identity gives a block decomposition with respect to the
first-level branching, and this in turn leads to a renormalization
formula for the scalar resolvent function 
\[
m\left(z\right)=\left\langle \phi,\left(zI-K_{\infty}\right)^{-1}\phi\right\rangle .
\]
Thus the operator $K_{\infty}$ carries both a filtration description
and a self-similar operator description, and these two descriptions
interact through the branching structure of the Cantor system.

\textbf{Literature context.} The present work sits at a meeting point
of several established lines of analysis and operator theory. At the
level of the finite operators $K_{m}$, the basic idea is a filtration
built from cylinder $\sigma$-algebras together with a Haar-type decomposition
adapted to a binary tree. In the symmetric case this places the problem
close to classical Haar systems, martingale difference methods, and
multiplier constructions associated with filtrations and conditional
expectations \cite{MR833073,MR402915,MR864712,MR618663,MR1228209}.
At the same time, because the underlying measure is a Bernoulli Cantor
measure, the relevant function system is also part of the broader
harmonic analysis of self-similar and fractal measures, including
orthogonal and wavelet-type constructions on Cantor and related spaces
\cite{MR1648441,MR1785282,MR3343663,MR1840042,MR2441420,MR1410793,MR2254554}.

A second theme is that the first-level splitting of the Cantor set
produces branch isometries and hence an operator-theoretic self-similarity.
This places the later parts of the paper near the literature on iterated
function systems, Cuntz families, and wavelet constructions generated
by isometries and self-similar branch maps \cite{MR467330,MR1469149,MR1743534,MR1998572,MR3103223}.
In that language, the limit operator $K_{\infty}$ may be viewed as
a concrete positive operator canonically generated by the cylinder
tree, while the fixed-point identity derived in \prettyref{sec:6}
expresses the same object intrinsically through the first-level self-similarity.
This passage from an explicit finite construction to an intrinsic
recursive operator identity is a central point of the paper.

There is also a broader operator-theoretic context. The affine map
defining $K_{\infty}$ has a completely positive linear part, so the
fixed-point characterization belongs naturally to the general framework
of completely bounded and completely positive maps, their spectra,
and their fixed-point spaces \cite{MR1976867,MR2572704,MR2805505,MR2839059,MR3059439,MR3990189,MR4126800,MR4478259,MR4550310,MR4718688}.
Our setting is much more concrete than that general theory, but it
shows that a simple self-similar measure space can produce a nontrivial
compact positive operator whose structure is simultaneously accessible
from frame theory, filtration methods, and completely positive recursion.

Finally, the block formulas and scalar resolvent recursion obtained
later in the paper are related in spirit to renormalization ideas
that occur elsewhere in analysis on self-similar sets and in spectral
problems for recursively generated operators \cite{MR1840042,MR2254554,MR2441420,MR4670347},
as well as to more classical resolvent and spectral comparison methods
in operator theory \cite{MR1335452,MR2154153,MR4369236}. What is
specific here is that the operator being renormalized is not introduced
abstractly: it arises directly from the frame operators of cylinder
indicators for the Bernoulli Cantor measure. The resulting picture
combines concrete finite-rank geometry, weighted Haar structure, norm-limit
operator theory, and self-similar recursion in a single model.

\textbf{Organization.} \prettyref{sec:2} fixes the Bernoulli Cantor
model, the cylinder notation, the associated filtration, and the finite
cylinder frame operators. \prettyref{sec:3} treats the equal-split
case. \prettyref{sec:4} develops the weighted finite level structure
and proves the filtration formula. \prettyref{sec:5} constructs the
norm-limit operator and gives a lower bound for its top eigenvalue.
\prettyref{sec:6} proves the self-similar characterization and derives
the block form, the completely positive Neumann expansion, and the
resolvent renormalization identity. \prettyref{sec:7} uses this block
decomposition to derive the scalar resolvent renormalization formula
and to reduce the top eigenvalue problem to a scalar equation for
the rooted resolvent function.

\section{Preliminaries}\label{sec:2}

In this section we fix the Bernoulli Cantor model and collect the
basic facts about cylinders, filtrations, and finite frame operators
that will be used throughout the paper.

\subsection{Cylinders and the Bernoulli measure}

Let $S_{0},S_{2}:\left[0,1\right]\rightarrow\left[0,1\right]$ be
the affine contractions
\[
S_{0}\left(x\right)=\frac{x}{3},\qquad S_{2}\left(x\right)=\frac{x+2}{3},
\]
so that $S_{0}$ maps $\left[0,1\right]$ onto the left third and
$S_{2}$ maps $\left[0,1\right]$ onto the right third. The middle-third
Cantor set $C$ is the unique nonempty compact set satisfying 
\[
C=S_{0}\left(C\right)\cup S_{2}\left(C\right).
\]
Equivalently, $C$ consists of those points in $\left[0,1\right]$
whose ternary expansion uses only the digits $0$ and $2$.

Fix $0<p<1$. The Bernoulli Cantor measure $\mu_{p}$ is the unique
Borel probability measure on $\left[0,1\right]$ satisfying 
\begin{equation}
\mu_{p}=p\mu_{p}\circ S^{-1}_{0}+\left(1-p\right)\mu_{p}\circ S^{-1}_{2},\label{eq:2-1a}
\end{equation}
and $\mathrm{supp}\left(\mu_{p}\right)=C$. Thus $\mu_{p}$ assigns
weight $p$ to the left first-level branch and weight $1-p$ to the
right first-level branch, and the same rule is repeated independently
at every later stage.

It is convenient to describe this measure by words. Let 
\[
\Sigma=\left\{ 0,2\right\} ^{\mathbb{N}}
\]
with the product $\sigma$-algebra, and for a finite word 
\[
w=\left(w_{1},\dots,w_{m}\right)\in\left\{ 0,2\right\} ^{m}
\]
write $\left|w\right|=m$. We also write $\varnothing$ for the empty
word, with $\left|\varnothing\right|=0$. 

For a finite word $w$ define 
\[
S_{w}=S_{w_{1}}\circ\cdots\circ S_{w_{m}},
\]
with the convention that $S_{\varnothing}$ is the identity map. By
definition, $S_{w_{n}}$ acts first, while the outermost map $S_{w_{1}}$
determines the first cylinder level. The corresponding triadic cylinder
interval is 
\[
I_{w}=S_{w}\left(\left[0,1\right]\right),
\]
and the associated cylinder subset of the Cantor set is 
\[
C_{w}=C\cap I_{w}=S_{w}\left(C\right).
\]
Thus $C_{\varnothing}=C$, while $C_{0}$ and $C_{2}$ are the two
first-level Cantor pieces, $C_{00},C_{02},C_{20},C_{22}$ are the
second-level pieces, and so on.

For the empty word $\varnothing$ we have 
\[
C_{\varnothing}=C,\qquad1_{C_{\varnothing}}=1_{C}=\phi,\qquad\mu_{p}\left(C_{\varnothing}\right)=1.
\]
When using the weighted sibling differences from \prettyref{sec:4},
we also write 
\[
h^{\left(p\right)}_{\varnothing}=\left(1-p\right)1_{C_{0}}-p\,1_{C_{2}}.
\]

If $N_{0}\left(w\right)$ and $N_{2}\left(w\right)$ denote the number
of $0$'s and $2$'s in $w$, then repeated use of the self-similarity
relation gives 
\[
\mu_{p}\left(C_{w}\right)=p^{N_{0}\left(w\right)}\left(1-p\right)^{N_{2}\left(w\right)}.
\]
In particular, 
\begin{equation}
\mu_{p}\left(C_{w0}\right)=p\mu_{p}\left(C_{w}\right),\qquad\mu_{p}\left(C_{w2}\right)=\left(1-p\right)\mu_{p}\left(C_{w}\right).\label{eq:2-2a}
\end{equation}
When $p=\frac{1}{2}$ this reduces to the uniform rule 
\begin{equation}
\mu_{1/2}\left(C_{w}\right)=2^{-\left|w\right|}.\label{eq:2-1}
\end{equation}

For each fixed level $m$, the family 
\[
\left\{ C_{w}:\left|w\right|=m\right\} 
\]
partitions $C$ up to $\mu_{p}$-null sets. The cylinders are nested
exactly as the word tree is nested: if $u$ and $v$ are finite words,
then 
\[
C_{u}\cap C_{v}=\begin{cases}
C_{v}, & \text{if \ensuremath{v} extends \ensuremath{u},}\\
C_{u}, & \text{if \ensuremath{u} extends \ensuremath{v},}\\
\varnothing, & \text{otherwise.}
\end{cases}
\]
Also, every cylinder splits into its two children: 
\[
C_{w}=C_{w0}\dot{\cup}C_{w2}\qquad\text{modulo \ensuremath{\mu_{p}}-null sets.}
\]

\subsection{Filtration, projections, and frame operators}

We now introduce the natural filtration. For each $n\ge0$, let 
\[
\mathcal{G}_{n}=\sigma\left(\left\{ C_{w}:\left|w\right|=n\right\} \right).
\]
Then $\left(\mathcal{G}_{n}\right)_{n\ge0}$ is an increasing filtration
on $\left(C,\mu_{p}\right)$, and $\mathcal{G}_{n}$ consists of those
measurable sets that are unions of level-$n$ cylinders, modulo null
sets. Let 
\begin{equation}
E_{n}:L^{2}\left(\mu_{p}\right)\to L^{2}\left(\mu_{p}\right)\label{eq:2-2}
\end{equation}
denote conditional expectation onto $\mathcal{G}_{n}$. Equivalently,
$E_{n}$ is the orthogonal projection onto the closed subspace of
functions that are constant on each level-$n$ cylinder. Thus, for
$f\in L^{2}\left(\mu_{p}\right)$ and $\left|w\right|=n$, 
\begin{equation}
\left(E_{n}f\right)\big|_{C_{w}}=\frac{1}{\mu_{p}\left(C_{w}\right)}\int_{C_{w}}fd\mu_{p}.\label{eq:2-3}
\end{equation}

We shall also use the level-$n$ mass operator 
\begin{equation}
D_{n}:L^{2}\left(\mu_{p}\right)\to L^{2}\left(\mu_{p}\right),\qquad D_{n}f=\sum_{\left|w\right|=n}\mu_{p}\left(C_{w}\right)1_{C_{w}}f.\label{eq:2-4}
\end{equation}
In other words, $D_{n}$ is multiplication by the function that takes
the constant value $\mu_{p}\left(C_{w}\right)$ on each level-$n$
cylinder $C_{w}$.

For $m\ge0$, set 
\begin{equation}
\mathcal{F}_{m}=span\left\{ 1_{C_{u}}:\left|u\right|\le m\right\} \subset L^{2}\left(\mu_{p}\right).\label{eq:2-5}
\end{equation}
This is the finite-dimensional subspace generated by the cylinder
indicators up to depth $m$.

It is useful to make the associated synthesis and analysis operators
explicit. Let 
\[
\Phi_{m}:\ell^{2}(\left\{ 0,2\right\} ^{\le m})\to L^{2}\left(\mu_{p}\right)
\]
be given by 
\[
\Phi_{m}c=\sum_{\left|u\right|\le m}c_{u}1_{C_{u}}.
\]
Thus $\Phi_{m}$ is the synthesis operator of the finite cylinder
family. Its adjoint 
\[
\Phi^{*}_{m}:L^{2}\left(\mu_{p}\right)\to\ell^{2}(\left\{ 0,2\right\} ^{\le m})
\]
is the corresponding analysis operator, given by 
\[
\left(\Phi^{*}_{m}f\right)_{u}=\left\langle 1_{C_{u}},f\right\rangle ,\qquad\left|u\right|\le m.
\]

We then define the finite cylinder frame operator by 
\begin{equation}
K_{m}=\Phi_{m}\Phi^{*}_{m},\label{eq:2-6}
\end{equation}
that is, 
\begin{equation}
K_{m}f=\sum_{\left|u\right|\le m}\left\langle 1_{C_{u}},f\right\rangle 1_{C_{u}},\qquad f\in L^{2}\left(\mu_{p}\right).\label{eq:2-7}
\end{equation}

The operator $K_{m}$ is therefore the frame operator of the finite
cylinder family, while $\Phi_{m}$ is its synthesis operator. Since
each $1_{C_{u}}$ belongs to $L^{2}\left(\mu_{p}\right)$, the operator
$K_{m}$ is finite-rank, positive, and self-adjoint.

A second viewpoint, which we use mainly for intuition, is the kernel
$K\left(A,B\right)=\mu_{p}\left(A\cap B\right)$ on the family of
cylinders. The associated feature map is simply $A\mapsto1_{A}\in L^{2}\left(\mu_{p}\right)$,
so that 
\[
K\left(A,B\right)=\left\langle 1_{A},1_{B}\right\rangle _{L^{2}\left(\mu_{p}\right)}.
\]
Thus the geometric intersection structure of cylinders is encoded
directly in the Hilbert space geometry of their indicator functions.
In the body of the paper, however, we will work primarily in the concrete
space $L^{2}\left(\mu_{p}\right)$ rather than in an abstract reproducing
kernel Hilbert space.

The equal-split case $p=\frac{1}{2}$ will be treated first as a model
case (\prettyref{sec:3}). There the cylinder masses depend only on
depth, and the corresponding Haar differences diagonalize the finite
operators $K_{m}$. For general $0<p<1$, the same tree of cylinders
remains in place, but the masses become branch-dependent, and the
resulting weighted structure is no longer diagonal (\prettyref{sec:4}).
The later sections analyze precisely how this loss of symmetry is
reflected in the finite level matrix structure, in the filtration
formula, and in the self-similar description of the limit operator
(Sections \ref{sec:5}--\ref{sec:7}).

\section{Symmetric case and Haar structure}\label{sec:3}

We treat the symmetric Bernoulli case $p=\frac{1}{2}$ and write $\mu:=\mu_{1/2}$.
In this case the cylinder masses depend only on depth, the tree is
uniform, and the sibling differences $h_{w}:=1_{C_{w0}}-1_{C_{w2}}$
form an orthogonal Haar system. The finite cylinder frame operators
$K_{m}$ are diagonal in this basis.

The same symmetry yields a filtration formula for $K_{m}$, which
we rewrite in terms of martingale differences. This provides the spectral
description of $K_{m}$ and of the limit operator $K_{\infty}$. In
\prettyref{sec:4}, the filtration formula remains valid in the weighted
case, but the scalar action on martingale layers no longer holds.

The Haar system and martingale difference decomposition used below
are classical; see, for example, \cite{MR833073,MR402915,MR864712,MR1228209}.
Operators built from conditional expectations along a filtration,
and their relation to multiplier theory, appear in \cite{MR864712,MR618663}.
We apply these tools to the cylinder frame operators introduced in
\prettyref{sec:2}.

We begin with the finite level diagonalization.
\begin{prop}
\label{prop:3-1}Let $\mathcal{F}_{m}$ be as in \eqref{eq:2-5}.
Let 
\[
\phi:=1_{C},\qquad h_{w}:=1_{C_{w0}}-1_{C_{w2}}\quad\left(|w|\le m-1\right).
\]

Then $\left\{ \phi\right\} \cup\left\{ h_{w}:|w|\le m-1\right\} $
is an orthogonal basis of $\mathcal{F}_{m}$, and the operator $K_{m}$
from \eqref{eq:2-6}--\eqref{eq:2-7} is diagonal in this basis.
More precisely, 
\begin{equation}
K_{m}\phi=\left(2-2^{-m}\right)\phi,\label{eq:3-1}
\end{equation}
and if $|w|=\ell$ with $0\le\ell\le m-1$, then 
\begin{equation}
K_{m}h_{w}=2^{-\ell}\left(1-2^{-\left(m-\ell\right)}\right)h_{w}.\label{eq:3-2}
\end{equation}

Equivalently, if $E_{n}$ denotes the orthogonal projection in \eqref{eq:2-2}--\eqref{eq:2-3},
then 
\begin{equation}
K_{m}=\sum^{m}_{n=0}2^{-n}E_{n}\label{eq:3-3}
\end{equation}
as an operator on $L^{2}\left(\mu\right)$. 
\end{prop}

\begin{proof}
The equal split property implies inductively that $\mu\left(C_{w}\right)=2^{-|w|}$,
$|w|\le m$. For each $w$ with $|w|\le m-1$, the children $C_{w0}$
and $C_{w2}$ are disjoint, hence 
\[
\left\langle h_{w},\phi\right\rangle =\mu\left(C_{w0}\right)-\mu\left(C_{w2}\right)=0,
\]
and also 
\[
\left\Vert h_{w}\right\Vert ^{2}=\mu\left(C_{w0}\right)+\mu\left(C_{w2}\right)=\mu\left(C_{w}\right)=2^{-|w|}.
\]
If $u\neq v$ are distinct words with $|u|,|v|\le m-1$, then $h_{u}$
and $h_{v}$ are supported on disjoint sets unless one word extends
the other. In the extension case, say $v$ extends $u$, the function
$h_{u}$ is constant on each atom at level $|u|+1$, while $h_{v}$
has mean zero on each such atom. In either case one checks 
\[
\left\langle h_{u},h_{v}\right\rangle =0.
\]
Thus $\phi$ is orthogonal to each $h_{w}$, and the family $\left\{ h_{w}:|w|\le m-1\right\} $
is pairwise orthogonal.

Next we show spanning. The identities 
\[
1_{C_{w}}=1_{C_{w0}}+1_{C_{w2}},\qquad h_{w}=1_{C_{w0}}-1_{C_{w2}}
\]
imply 
\[
1_{C_{w0}}=\frac{1}{2}\left(1_{C_{w}}+h_{w}\right),\qquad1_{C_{w2}}=\frac{1}{2}\left(1_{C_{w}}-h_{w}\right).
\]
Starting from $\phi=1_{C}$ at level $0$ and iterating these relations,
every indicator $1_{C_{u}}$ with $|u|\le m$ lies in the span of
$\phi$ and the Haar differences $\left\{ h_{w}:|w|\le m-1\right\} $.
Hence these vectors span $\mathcal{F}_{m}$. Since they are orthogonal
and their cardinality is 
\[
1+\sum^{m-1}_{\ell=0}2^{\ell}=2^{m},
\]
which equals $\dim\mathcal{F}_{m}$, they form an orthogonal basis
of $\mathcal{F}_{m}$.

We now compute the action of $K_{m}$ on this basis.

For $\phi=1_{C}$, we have $\left\langle \phi,1_{C_{u}}\right\rangle =\mu\left(C_{u}\right)=2^{-|u|}$,
hence 
\[
K_{m}\phi=\sum_{|u|\le m}2^{-|u|}1_{C_{u}}.
\]
Fix $x\in C$. For each $r=0,1,\dots,m$, exactly one atom $C_{u}$
with $|u|=r$ contains $x$ (modulo null sets). Therefore 
\[
\left(K_{m}\phi\right)\left(x\right)=\sum^{m}_{r=0}2^{-r},
\]
and thus $K_{m}\phi=\left(\sum^{m}_{r=0}2^{-r}\right)\phi=\left(2-2^{-m}\right)\phi$,
which is \prettyref{eq:3-1}. 

Now fix $w$ with $|w|=\ell\le m-1$. For any word $u$ with $|u|\le m$,
\[
\left\langle h_{w},1_{C_{u}}\right\rangle =\mu\left(C_{u}\cap C_{w0}\right)-\mu\left(C_{u}\cap C_{w2}\right).
\]
If $u$ is a prefix of $w$, then $C_{w0},C_{w2}\subset C_{u}$, hence
\[
\left\langle h_{w},1_{C_{u}}\right\rangle =\mu\left(C_{w0}\right)-\mu\left(C_{w2}\right)=0
\]
by equal split. If $u$ extends $w0$, then $C_{u}\subset C_{w0}$
and $C_{u}\cap C_{w2}=\varnothing$, hence 
\[
\left\langle h_{w},1_{C_{u}}\right\rangle =\mu\left(C_{u}\right)=2^{-|u|}.
\]
If $u$ extends $w2$, similarly $\left\langle h_{w},1_{C_{u}}\right\rangle =-2^{-|u|}$.
In all other cases, $C_{u}$ is disjoint from both children and the
inner product is $0$. Therefore 
\[
K_{m}h_{w}=\sum_{\substack{|u|\le m\\
u\succeq w0
}
}2^{-|u|}1_{C_{u}}-\sum_{\substack{|u|\le m\\
u\succeq w2
}
}2^{-|u|}1_{C_{u}}.
\]
Evaluating at $x\in C$, if $x\in C_{w0}$ then for each depth $r=\ell+1,\dots,m$
there is exactly one word $u$ of length $r$ extending $w0$ with
$x\in C_{u}$, and no word extending $w2$ contains $x$. Hence 
\[
\left(K_{m}h_{w}\right)\left(x\right)=\sum^{m}_{r=\ell+1}2^{-r}.
\]
If $x\in C_{w2}$, the same argument gives $\left(K_{m}h_{w}\right)\left(x\right)=-\sum^{m}_{r=\ell+1}2^{-r}$,
and if $x\notin C_{w}$ then both sums vanish and $\left(K_{m}h_{w}\right)\left(x\right)=0$.
In all cases, 
\[
K_{m}h_{w}=\left(\sum^{m}_{r=\ell+1}2^{-r}\right)h_{w}=2^{-\ell}\left(1-2^{-\left(m-\ell\right)}\right)h_{w}
\]
This is \prettyref{eq:3-2}. 

Finally, define $E_{n}$ as the orthogonal projection onto the subspace
of functions constant on each atom $\left\{ C_{u}:|u|=n\right\} $.
Then $E_{n}\phi=\phi$ for all $n$, and for $|w|=\ell$ one has 
\[
E_{n}h_{w}=\begin{cases}
0, & n\le\ell,\\
h_{w}, & n\ge\ell+1,
\end{cases}
\]
since $h_{w}$ is constant on each level-$n$ atom when $n\ge\ell+1$
and has mean zero on each level-$\ell$ atom. Thus 
\[
\left(\sum^{m}_{n=0}2^{-n}E_{n}\right)\phi=\left(\sum^{m}_{n=0}2^{-n}\right)\phi,\qquad\left(\sum^{m}_{n=0}2^{-n}E_{n}\right)h_{w}=\left(\sum^{m}_{r=\ell+1}2^{-r}\right)h_{w},
\]
which matches the already computed action of $K_{m}$ on the orthogonal
basis of $\mathcal{F}_{m}$. Both operators vanish on $\mathcal{F}^{\perp}_{m}$,
hence \prettyref{eq:3-3} holds. 
\end{proof}

The identity $K_{m}=\sum^{m}_{n=0}2^{-n}E_{n}$ places the operator
on the cylinder filtration. We now express this formula in terms of
martingale difference projections.
\begin{prop}
\label{prop:3-4} Let $E_{n}:L^{2}\left(\mu\right)\to L^{2}\left(\mu\right)$
be the conditional expectation as in \eqref{eq:2-2}--\eqref{eq:2-3},
and set 
\begin{equation}
d_{0}=E_{0},\qquad d_{n}=E_{n}-E_{n-1}\qquad\left(n\ge1\right).\label{eq:3-4}
\end{equation}
Then, for every $m\ge0$, 
\begin{equation}
K_{m}=\left(2-2^{-m}\right)d_{0}+\sum^{m}_{n=1}2^{-\left(n-1\right)}\left(1-2^{-\left(m-n+1\right)}\right)d_{n}.\label{eq:3-5}
\end{equation}
In particular, each martingale difference subspace 
\[
\mathcal{D}_{0}=E_{0}L^{2}\left(\mu\right)=\mathrm{span}\left\{ \phi\right\} ,\qquad\mathcal{D}_{n}=d_{n}L^{2}\left(\mu\right)\qquad\left(n\ge1\right)
\]
is $K_{m}$-invariant, and $K_{m}$ acts on $\mathcal{D}_{n}$ as
scalar multiplication by 
\[
\lambda_{n,m}=\begin{cases}
2-2^{-m}, & n=0,\\[0.8ex]
2^{-\left(n-1\right)}\left(1-2^{-\left(m-n+1\right)}\right), & 1\le n\le m,
\end{cases}
\]
while $K_{m}$ vanishes on $\mathcal{D}_{n}$ for $n\ge m+1$. 
\end{prop}

\begin{proof}
By \prettyref{eq:3-4}, we get $E_{n}=\sum^{n}_{j=0}d_{j}$, so the
identity \prettyref{eq:3-3} can be written as 
\begin{multline*}
K_{m}=\sum^{m}_{n=0}2^{-n}E_{n}=\sum^{m}_{n=0}2^{-n}\sum^{n}_{j=0}d_{j}\\
=\sum^{m}_{j=0}\left(\sum^{m}_{n=j}2^{-n}\right)d_{j}=\left(\sum^{m}_{r=0}2^{-r}\right)d_{0}+\sum^{m}_{n=1}\left(\sum^{m}_{r=n}2^{-r}\right)d_{n}.
\end{multline*}
This gives the martingale difference formula \prettyref{eq:3-5}. 

Therefore $K_{m}$ acts by the scalar $\lambda_{n,m}$ on $\mathcal{D}_{n}$,
and by $0$ on $\mathcal{D}_{n}$ for $n\ge m+1$. 
\end{proof}

In this form, $K_{m}$ acts by a scalar on each martingale difference
subspace. The next corollary gives the decomposition and multiplicities.
\begin{cor}
\label{cor:3-5} For each $m\ge0$, one has
\[
L^{2}\left(\mu\right)=\bigoplus^{\infty}_{n=0}\mathcal{D}_{n},\qquad K_{m}=\bigoplus^{m}_{n=0}\lambda_{n,m}I_{\mathcal{D}_{n}}\oplus0
\]
with $\lambda_{n,m}$ as in \prettyref{prop:3-4}. Moreover, 
\[
\dim\mathcal{D}_{0}=1,\qquad\dim\mathcal{D}_{n}=2^{n-1}\qquad\left(n\ge1\right).
\]
Hence the nonzero spectrum of $K_{m}$ is 
\[
\left\{ 2-2^{-m}\right\} \cup\left\{ 2^{-\left(n-1\right)}\left(1-2^{-\left(m-n+1\right)}\right):1\le n\le m\right\} ,
\]
where the eigenvalue 
\[
2^{-\left(n-1\right)}\left(1-2^{-\left(m-n+1\right)}\right)
\]
has multiplicity $2^{n-1}$. 
\end{cor}

\begin{proof}
The orthogonal decomposition by martingale differences is standard
for the filtration $\left(\mathcal{G}_{n}\right)_{n\ge0}$. The diagonal
form of $K_{m}$ follows from \prettyref{prop:3-4}. It remains to
compute $\dim\mathcal{D}_{n}$.

Since $E_{n}L^{2}\left(\mu\right)$ is the space of functions constant
on each level-$n$ cylinder, we have $\dim E_{n}L^{2}\left(\mu\right)=2^{n}$.
Therefore $\dim\mathcal{D}_{0}=1$, and 
\[
\dim\mathcal{D}_{n}=\dim E_{n}L^{2}\left(\mu\right)-\dim E_{n-1}L^{2}\left(\mu\right)=2^{n}-2^{n-1}=2^{n-1}
\]
for $n\ge1$. The spectral multiplicities follow immediately. 
\end{proof}

Since the coefficients are summable, the same description passes to
the norm-limit operator.
\begin{cor}
\label{cor:3-6} The series 
\begin{equation}
K_{\infty}=\sum^{\infty}_{n=0}2^{-n}E_{n}\label{eq:3-6}
\end{equation}
converges in operator norm, and 
\begin{equation}
K_{\infty}=2d_{0}+\sum^{\infty}_{n=1}2^{-\left(n-1\right)}d_{n}.\label{eq:3-7}
\end{equation}

Hence $K_{\infty}$ acts on $\mathcal{D}_{0}$ as multiplication by
$2$, and on $\mathcal{D}_{n}$ as multiplication by $2^{-\left(n-1\right)}$,
$n\ge1$. Therefore the spectrum of $K_{\infty}$ is 
\[
\sigma\left(K_{\infty}\right)=\left\{ 2\right\} \cup\left\{ 2^{-\left(n-1\right)}:n\ge1\right\} \cup\left\{ 0\right\} ,
\]
where $2$ has multiplicity $1$ and $2^{-\left(n-1\right)}$ has
multiplicity $2^{n-1}$ for $n\ge1$. In particular, 
\[
\left\Vert K_{\infty}\right\Vert =2.
\]
\end{cor}

\begin{proof}
Since $\left\Vert E_{n}\right\Vert =1$ for all $n$, 
\[
\sum^{\infty}_{n=0}\left\Vert 2^{-n}E_{n}\right\Vert \le\sum^{\infty}_{n=0}2^{-n}<\infty,
\]
so the series \prettyref{eq:3-6} converges in operator norm. Passing
to the limit in \prettyref{prop:3-4} gives \prettyref{eq:3-7}.

The spectral statements follow immediately from the orthogonal decomposition
$L^{2}\left(\mu\right)=\bigoplus^{\infty}_{n=0}\mathcal{D}_{n}$.
The norm is the top eigenvalue, namely $2$. 
\end{proof}

The spectral description of $K_{\infty}$ allows one to determine
its position among the Schatten classes. Recall that for $1\le r<\infty$,
the Schatten class $S_{r}$ consists of all compact operators $T$
on a Hilbert space such that $\sum s_{j}(T)^{r}<\infty$, where $\left(s_{j}(T)\right)$
are the singular values of $T$; see \cite{MR2154153}. In particular,
$S_{2}$ is the Hilbert-Schmidt class and $S_{1}$ is the trace class.
\begin{cor}
\label{cor:3-7} For the symmetric limit operator $K_{\infty}$, one
has 
\[
K_{\infty}\in S_{r}\Longleftrightarrow r>1.
\]
In particular, $K_{\infty}$ is Hilbert-Schmidt but not trace class. 
\end{cor}

\begin{proof}
By \prettyref{cor:3-6}, the nonzero eigenvalues of $K_{\infty}$
are $2$ with multiplicity $1$, and $2^{-\left(n-1\right)}$ with
multiplicity $2^{n-1}$ for $n\ge1$. Therefore 
\[
\sum s_{j}\left(K_{\infty}\right)^{r}=2^{r}+\sum^{\infty}_{n=1}2^{n-1}2^{-r\left(n-1\right)}=2^{r}+\sum^{\infty}_{n=1}2^{\left(1-r\right)\left(n-1\right)}.
\]
This converges if and only if $r>1$. Taking $r=2$ and $r=1$ gives
the final statement. 
\end{proof}

\begin{rem}
The symmetric case $p=\frac{1}{2}$ provides a solvable model for
the cylinder frame operators. Because the cylinder masses depend only
on depth, the operators are determined by the filtration structure
and act by scalars on the martingale difference subspaces. This reduces
the spectral analysis of $K_{m}$ and of the limit operator $K_{\infty}$
to geometric series on the tree, giving explicit eigenvalues, multiplicities,
and the threshold for Schatten class inclusion. 

In \prettyref{sec:4}, the weighted case retains the filtration formula
but loses the scalar action on martingale layers. The symmetric model
thus serves as a reference point for the structural changes that occur
when the branch weights are unequal.
\end{rem}

\section{Weighted case and filtration}\label{sec:4}

We now fix $0<p<1$ and return to the finite operators $K_{m}$ in
the nonsymmetric Bernoulli case. The equal-split diagonalization from
\prettyref{sec:3} fails, but two pieces of structure remain. First,
there is still a natural orthogonal family of sibling differences,
now weighted by the branch masses. Second, the cylinder filtration
from \prettyref{sec:2} still governs the operator through a conditional
expectation representation.

The main point of this section is to make these two descriptions explicit.
\prettyref{thm:4-1} identifies a weighted Haar basis for $\mathcal{F}_{m}$
and computes the resulting matrix coefficients, showing that the difference--difference
block is sparse in a tree sense (it vanishes unless the indices are
comparable). Afterward we state the filtration formula $K_{m}=\sum^{m}_{n=0}D_{n}E_{n}$,
which will be the bridge to the norm limit analysis in \prettyref{sec:5}.
\begin{thm}
\label{thm:4-1}Fix $0<p<1$, let $\mu_{p}$ be the Bernoulli Cantor
measure as in \prettyref{eq:2-1a} so that \prettyref{eq:2-2a} holds.
Fix $m\ge1$. Let $\mathcal{F}_{m}$ be as in \prettyref{eq:2-5},
and let $K_{m}$ be the corresponding frame operator. For each word
$w$ with $|w|\le m-1$, define the weighted sibling difference 
\[
h^{\left(p\right)}_{w}=\left(1-p\right)1_{C_{w0}}-p\,1_{C_{w2}}.
\]
Also write 
\[
\phi=1_{C},\qquad q=p^{2}+\left(1-p\right)^{2}.
\]
Then the following hold.
\begin{enumerate}
\item The family 
\begin{equation}
\left\{ \phi\right\} \cup\left\{ h^{\left(p\right)}_{w}:|w|\le m-1\right\} \label{eq:4-1}
\end{equation}
is an orthogonal basis of $\mathcal{F}_{m}$. Moreover, 
\begin{equation}
\Vert h^{\left(p\right)}_{w}\Vert^{2}=p\left(1-p\right)\mu_{p}\left(C_{w}\right).\label{eq:4-2}
\end{equation}
\item The operator $K_{m}$ leaves $\mathcal{F}_{m}$ invariant and vanishes
on $\mathcal{F}^{\perp}_{m}$.
\item For every word $w$ with $|w|\le m-1$, 
\begin{equation}
\left\langle h^{\left(p\right)}_{w},K_{m}\phi\right\rangle =p\left(1-p\right)\left(2p-1\right)\mu_{p}\left(C_{w}\right)^{2}\sum^{m-|w|-1}_{j=0}q^{j}.\label{eq:4-3}
\end{equation}
\item For every word $w$ with $|w|\le m-1$, 
\begin{equation}
\left\langle h^{\left(p\right)}_{w},K_{m}h^{\left(p\right)}_{w}\right\rangle =2p^{2}\left(1-p\right)^{2}\mu_{p}\left(C_{w}\right)^{2}\sum^{m-|w|-1}_{j=0}q^{j}.\label{eq:4-4}
\end{equation}
\item Let $u,v$ be words with $|u|,|v|\le m-1$, $u\neq v$.

If $u$ and $v$ are incomparable, then 
\[
\left\langle h^{\left(p\right)}_{u},K_{m}h^{\left(p\right)}_{v}\right\rangle =0.
\]

If $u$ is a proper prefix of $v$, then 
\[
\left\langle h^{\left(p\right)}_{u},K_{m}h^{\left(p\right)}_{v}\right\rangle =\begin{cases}
p\left(1-p\right)^{2}\left(2p-1\right)\mu_{p}\left(C_{v}\right)^{2}\sum^{m-|v|-1}_{j=0}q^{j}, & \text{if \ensuremath{v} begins with \ensuremath{u0},}\\[1.2ex]
-p^{2}\left(1-p\right)\left(2p-1\right)\mu_{p}\left(C_{v}\right)^{2}\sum^{m-|v|-1}_{j=0}q^{j}, & \text{if \ensuremath{v} begins with \ensuremath{u2}.}
\end{cases}
\]

\end{enumerate}
Consequently, in the orthonormal basis 
\begin{equation}
e_{\phi}=\phi,\qquad e_{w}=\frac{h^{\left(p\right)}_{w}}{\sqrt{p\left(1-p\right)\mu_{p}\left(C_{w}\right)}},\label{eq:4-5}
\end{equation}
the matrix of $K_{m}|_{\mathcal{F}_{m}}$ is self-adjoint, and its
difference-difference block is tree-banded: the entry 
\[
\left\langle e_{v},K_{m}e_{u}\right\rangle 
\]
vanishes unless $u$ and $v$ are comparable. In particular, $p=\frac{1}{2}$
is the unique parameter for which this matrix is diagonal in the weighted
Haar basis. 
\end{thm}

\begin{proof}
To simplify notations, we write 
\[
1_{w}:=1_{C_{w}},\qquad\mu(w):=\mu_{p}\left(C_{w}\right).
\]

We first prove $(1)$.

For each $w$ with $|w|\le m-1$, $h^{\left(p\right)}_{w}=\left(1-p\right)1_{w0}-p1_{w2}$.
Using $C_{w0}\cap C_{w2}=\varnothing$ and \prettyref{eq:2-2a}, we
get 
\[
\Vert h^{\left(p\right)}_{w}\Vert^{2}=\left(1-p\right)^{2}\mu\left(w0\right)+p^{2}\mu\left(w2\right)=p\left(1-p\right)\mu\left(w\right),
\]
which is \prettyref{eq:4-2}. Also, 
\[
\langle h^{\left(p\right)}_{w},\phi\rangle=\left(1-p\right)\mu\left(w0\right)-p\mu\left(w2\right)=0.
\]

Now let $u\neq v$ with $|u|,|v|\le m-1$. If $u$ and $v$ are incomparable,
then the supports of $h^{\left(p\right)}_{u}$ and $h^{\left(p\right)}_{v}$
are disjoint, so $\langle h^{\left(p\right)}_{u},h^{\left(p\right)}_{v}\rangle=0$.
If $u$ is a proper prefix of $v$, then $h^{\left(p\right)}_{u}$
is constant on $C_{v}$, while 
\[
\int_{C_{v}}h^{\left(p\right)}_{v}\,d\mu_{p}=\left(1-p\right)\mu\left(v0\right)-p\mu\left(v2\right)=0.
\]
Hence $\langle h^{\left(p\right)}_{u},h^{\left(p\right)}_{v}\rangle=0$.
The case $v\prec u$ is the same. Thus \prettyref{eq:4-1} is an orthogonal
family.

To prove spanning, note that 
\[
1_{w0}=h^{\left(p\right)}_{w}+p\,1_{w},\qquad1_{w2}=\left(1-p\right)1_{w}-h^{\left(p\right)}_{w}.
\]
Starting from $\phi=1_{C}$, these identities show inductively that
every $1_{u}$ with $|u|\le m$ lies in the span of 
\[
\left\{ \phi\right\} \cup\left\{ h^{\left(p\right)}_{w}:|w|\le m-1\right\} .
\]
Since $1+\sum^{m-1}_{\ell=0}2^{\ell}=2^{m}=\dim\mathcal{F}_{m}$,
the family \prettyref{eq:4-1} is an orthogonal basis of $\mathcal{F}_{m}$.

This also gives $(2)$: each $1_{u}$ lies in $\mathcal{F}_{m}$,
hence $K_{m}$ maps $\mathcal{F}_{m}$ into itself, and if $f\in\mathcal{F}^{\perp}_{m}$,
then $\left\langle 1_{u},f\right\rangle =0$ for all $|u|\le m$,
so $K_{m}f=0$.

To prove $(3)$, $(4)$ and $(5)$, we need the following lemma.
\end{proof}

\begin{lem}
For every word $z$ and every $n\ge0$, 
\[
\sum_{\substack{u\succeq z\\
|u|=|z|+n
}
}\mu\left(u\right)^{2}=\mu\left(z\right)^{2}q^{n},\qquad q:=p^{2}+\left(1-p\right)^{2}.
\]
\end{lem}

\begin{proof}
\noindent For $n=0$ this is immediate. Suppose it holds for $n$.
Then 
\[
\sum_{\substack{u\succeq z\\
|u|=|z|+n+1
}
}\mu\left(u\right)^{2}=\sum_{\substack{v\succeq z\\
|v|=|z|+n
}
}\left(\mu\left(v0\right)^{2}+\mu\left(v2\right)^{2}\right).
\]
Since 
\[
\mu\left(v0\right)^{2}+\mu\left(v2\right)^{2}=\left(p^{2}+\left(1-p\right)^{2}\right)\mu\left(v\right)^{2}=q\mu\left(v\right)^{2},
\]
the induction hypothesis gives 
\[
\sum_{\substack{u\succeq z\\
|u|=|z|+n+1
}
}\mu\left(u\right)^{2}=q\sum_{\substack{v\succeq z\\
|v|=|z|+n
}
}\mu\left(v\right)^{2}=q\mu\left(z\right)^{2}q^{n}=\mu\left(z\right)^{2}q^{n+1}.
\]
This proves the lemma. 
\end{proof}

\begin{proof}[Proof of \prettyref{thm:4-1} continued]
We now use 
\[
\left\langle f,K_{m}g\right\rangle =\sum_{|u|\le m}\left\langle f,1_{u}\right\rangle \left\langle 1_{u},g\right\rangle .
\]
A direct inspection gives, for $|w|\le m-1$, 
\[
\langle1_{u},h^{\left(p\right)}_{w}\rangle=\begin{cases}
\left(1-p\right)\mu\left(u\right), & \text{if \ensuremath{u\succeq w0},}\\[0.8ex]
-p\mu\left(u\right), & \text{if \ensuremath{u\succeq w2},}\\[0.8ex]
0, & \text{otherwise.}
\end{cases}
\]

We prove $(3)$. Since $\left\langle \phi,1_{u}\right\rangle =\mu(u)$,
\[
\langle h^{\left(p\right)}_{w},K_{m}\phi\rangle=\sum_{|u|\le m}\mu\left(u\right)\langle1_{u},h^{\left(p\right)}_{w}\rangle.
\]
Only descendants of $w0$ and $w2$ contribute, so 
\[
\langle h^{\left(p\right)}_{w},K_{m}\phi\rangle=\left(1-p\right)\sum_{\substack{u\succeq w0\\
|u|\le m
}
}\mu\left(u\right)^{2}-p\sum_{\substack{u\succeq w2\\
|u|\le m
}
}\mu\left(u\right)^{2}.
\]
Split by depth: 
\[
\sum_{\substack{u\succeq w0\\
|u|\le m
}
}\mu\left(u\right)^{2}=\sum^{m-|w|-1}_{j=0}\sum_{\substack{u\succeq w0\\
|u|=|w|+1+j
}
}\mu\left(u\right)^{2}=\mu\left(w0\right)^{2}\sum^{m-|w|-1}_{j=0}q^{j}.
\]
Similarly, 
\[
\sum_{\substack{u\succeq w2\\
|u|\le m
}
}\mu\left(u\right)^{2}=\mu\left(w2\right)^{2}\sum^{m-|w|-1}_{j=0}q^{j}.
\]
Since $\mu\left(w0\right)=p\mu\left(w\right)$ and $\mu\left(w2\right)=\left(1-p\right)\mu\left(w\right)$,
\begin{align*}
\langle h^{\left(p\right)}_{w},K_{m}\phi\rangle & =\left(\left(1-p\right)p^{2}-p\left(1-p\right)^{2}\right)\mu\left(w\right)^{2}\sum^{m-|w|-1}_{j=0}q^{j}\\
 & =p\left(1-p\right)\left(2p-1\right)\mu(w)^{2}\sum^{m-|w|-1}_{j=0}q^{j}
\end{align*}
which is \prettyref{eq:4-3}.

We prove $(4)$. Using the same coefficient formula, 
\[
\langle h^{\left(p\right)}_{w},K_{m}h^{\left(p\right)}_{w}\rangle=\sum_{|u|\le m}\left|\langle1_{u},h^{\left(p\right)}_{w}\rangle\right|^{2}.
\]
Thus 
\[
\langle h^{\left(p\right)}_{w},K_{m}h^{\left(p\right)}_{w}\rangle=\left(1-p\right)^{2}\sum_{\substack{u\succeq w0\\
|u|\le m
}
}\mu\left(u\right)^{2}+p^{2}\sum_{\substack{u\succeq w2\\
|u|\le m
}
}\mu\left(u\right)^{2}.
\]
Applying the lemma on the two child branches gives 
\begin{align*}
\langle h^{\left(p\right)}_{w},K_{m}h^{\left(p\right)}_{w}\rangle & =\left(\left(1-p\right)^{2}p^{2}+p^{2}\left(1-p\right)^{2}\right)\mu\left(w\right)^{2}\sum^{m-|w|-1}_{j=0}q^{j}\\
 & =2p^{2}\left(1-p\right)^{2}\mu(w)^{2}\sum^{m-|w|-1}_{j=0}q^{j}.
\end{align*}
hence \prettyref{eq:4-4} holds. 

We prove $(5)$. Let $u\neq v$.

If $u$ and $v$ are incomparable, then no cylinder $C_{t}$ can be
simultaneously contained in a descendant of $u0$ or $u2$ and in
a descendant of $v0$ or $v2$. Hence for every $t$, $\langle h^{\left(p\right)}_{u},1_{t}\rangle\langle1_{t},h^{\left(p\right)}_{v}\rangle=0$,
and therefore $\langle h^{\left(p\right)}_{u},K_{m}h^{\left(p\right)}_{v}\rangle=0$.

Now suppose $u$ is a proper prefix of $v$. First assume that the
first symbol after $u$ in $v$ is $0$, so every descendant of $v$
lies inside $u0$. Then for every $t\succeq v0$ or $t\succeq v2$,
$\langle h^{\left(p\right)}_{u},1_{t}\rangle=\left(1-p\right)\mu\left(t\right)$.
Hence 
\[
\langle h^{\left(p\right)}_{u},K_{m}h^{\left(p\right)}_{v}\rangle=\left(1-p\right)^{2}\sum_{\substack{t\succeq v0\\
|t|\le m
}
}\mu\left(t\right)^{2}-p\left(1-p\right)\sum_{\substack{t\succeq v2\\
|t|\le m
}
}\mu\left(t\right)^{2}.
\]
Using the lemma and the identities $\mu\left(v0\right)=p\mu\left(v\right)$,
$\mu\left(v2\right)=\left(1-p\right)\mu\left(v\right)$, we get 
\begin{align*}
\langle h^{\left(p\right)}_{u},K_{m}h^{\left(p\right)}_{v}\rangle & =\left(\left(1-p\right)^{2}p^{2}-p\left(1-p\right)^{3}\right)\mu\left(v\right)^{2}\sum^{m-|v|-1}_{j=0}q^{j}\\
 & =p\left(1-p\right)^{2}\left(2p-1\right)\mu\left(v\right)^{2}\sum^{m-|v|-1}_{j=0}q^{j}.
\end{align*}

If instead the first symbol after $u$ in $v$ is $2$, then every
descendant of $v$ lies inside $u2$, and so $\langle h^{\left(p\right)}_{u},1_{t}\rangle=-p\mu\left(t\right)$
for every $t\succeq v0$ or $t\succeq v2$. Therefore 
\[
\langle h^{\left(p\right)}_{u},K_{m}h^{\left(p\right)}_{v}\rangle=-p\left(1-p\right)\sum_{\substack{t\succeq v0\\
|t|\le m
}
}\mu(t)^{2}+p^{2}\sum_{\substack{t\succeq v2\\
|t|\le m
}
}\mu(t)^{2}.
\]
Applying the lemma again gives 
\[
\langle h^{\left(p\right)}_{u},K_{m}h^{\left(p\right)}_{v}\rangle=-p^{2}\left(1-p\right)\left(2p-1\right)\mu(v)^{2}\sum^{m-|v|-1}_{j=0}q^{j}.
\]

The final statement now follows immediately. In the orthonormal basis
\[
e_{\phi}=\phi,\qquad e_{w}=\frac{h^{\left(p\right)}_{w}}{\sqrt{p\left(1-p\right)\mu(w)}},
\]
the matrix of $K_{m}|_{\mathcal{F}_{m}}$ is real symmetric. By $(5)$,
the difference-difference block vanishes unless the two indices are
comparable, so it is tree-banded.

Finally, if $p=\frac{1}{2}$, then all off-diagonal coefficients in
$(3)$ and $(5)$ vanish because of the factor $2p-1$, so the weighted
Haar basis diagonalizes $K_{m}$. Conversely, if $p\neq\frac{1}{2}$,
then taking $w=\varnothing$ in $(3)$ gives 
\[
\langle h^{\left(p\right)}_{\varnothing},K_{m}\phi\rangle=p\left(1-p\right)\left(2p-1\right)\sum^{m-1}_{j=0}q^{j}\neq0,
\]
since $0<p<1$ and $\sum^{m-1}_{j=0}q^{j}>0$. Hence the matrix is
not diagonal. Therefore $p=\frac{1}{2}$ is the unique parameter for
which $K_{m}$ is diagonal in the weighted Haar basis.
\end{proof}

\prettyref{thm:4-1} gives a finite level orthogonal basis adapted
to the weighted tree and shows that $K_{m}\big|_{\mathcal{F}_{m}}$
is sparse in that basis, with nonzero entries only along comparable
chains. A second description of $K_{m}$ comes from the cylinder filtration:
it expresses $K_{m}$ as a sum of level-wise contributions built from
conditional expectation and the corresponding mass multipliers.
\begin{prop}
\label{prop:4-3} For $m\ge0$, the finite cylinder frame operator
satisfies 
\[
K_{m}=\sum^{m}_{n=0}D_{n}E_{n}
\]
on $L^{2}\left(\mu_{p}\right)$, where $E_{n}$ and $D_{n}$ are the
level-$n$ conditional expectation and mass operators from \prettyref{sec:2}. 
\end{prop}

\begin{proof}
Fix $f\in L^{2}\left(\mu_{p}\right)$. By definition, $K_{m}f=\sum^{m}_{n=0}\sum_{|u|=n}\left\langle 1_{C_{u}},f\right\rangle 1_{C_{u}}$.
For each word $u$ with $|u|=n$, 
\[
\left\langle 1_{C_{u}},f\right\rangle =\int_{C_{u}}fd\mu_{p}.
\]
Since $E_{n}f$ is constant on $C_{u}$, \prettyref{eq:2-3} gives
\[
\left(E_{n}f\right)\big|_{C_{u}}=\frac{1}{\mu_{p}\left(C_{u}\right)}\int_{C_{u}}fd\mu_{p}.
\]
Hence 
\[
\left\langle 1_{C_{u}},f\right\rangle 1_{C_{u}}=\mu_{p}\left(C_{u}\right)1_{C_{u}}E_{n}f.
\]
Summing over all words $u$ with $|u|=n$ yields 
\[
\sum_{|u|=n}\left\langle 1_{C_{u}},f\right\rangle 1_{C_{u}}=\sum_{|u|=n}\mu_{p}\left(C_{u}\right)1_{C_{u}}E_{n}f=D_{n}E_{n}f.
\]
Now sum in $n$ from $0$ to $m$. 
\end{proof}

The filtration representation makes the geometric support transparent.
In particular, it gives a short proof of the incomparable vanishing
in the weighted Haar basis by combining the projection identities
for $E_{n}$ with disjointness of cylinder supports.
\begin{cor}
\label{cor:4-4} In the orthonormal basis \prettyref{eq:4-5}, the
following gives an alternative structural proof of the incomparable
vanishing in \prettyref{thm:4-1}: if $u$ and $v$ are incomparable,
then 
\begin{equation}
\left\langle e_{u},K_{m}e_{v}\right\rangle =0.\label{eq:4-6}
\end{equation}
\end{cor}

\begin{proof}
Let $u$ and $v$ be incomparable words. By \prettyref{prop:4-3},
\[
\langle h^{\left(p\right)}_{v},K_{m}h^{\left(p\right)}_{u}\rangle=\sum^{m}_{n=0}\langle h^{\left(p\right)}_{v},D_{n}E_{n}h^{\left(p\right)}_{u}\rangle.
\]
If $n\le|u|$, then $E_{n}h^{\left(p\right)}_{u}=0$, and if $n\ge|u|+1$,
then $E_{n}h^{\left(p\right)}_{u}=h^{\left(p\right)}_{u}$. Hence
\[
\langle h^{\left(p\right)}_{v},K_{m}h^{\left(p\right)}_{u}\rangle=\sum^{m}_{n=|u|+1}\langle h^{\left(p\right)}_{v},D_{n}h^{\left(p\right)}_{u}\rangle.
\]
Since $D_{n}$ is a multiplication operator, $D_{n}h^{\left(p\right)}_{u}$
is supported in $C_{u}$, while $h^{\left(p\right)}_{v}$ is supported
in $C_{v}$. As $u$ and $v$ are incomparable, $C_{u}\cap C_{v}=\varnothing$,
so every inner product in the sum vanishes. Therefore \prettyref{eq:4-6}
holds. 
\end{proof}

For later use it is convenient to give the resulting matrix coefficients
in the orthonormal basis \prettyref{eq:4-5}. The next corollary simply
restates the formulas of \prettyref{thm:4-1} after normalization,
together with the vanishing from \prettyref{cor:4-4}.
\begin{cor}
\label{cor:4-5} Fix $0<p<1$ and $m\ge1$, and let $K_{m}$, $\phi$,
and $\left\{ e_{\phi}\right\} \cup\left\{ e_{w}\right\} $ be as in
\prettyref{thm:4-1} and \prettyref{eq:4-5}. Write 
\[
\mu(w)=\mu_{p}\left(C_{w}\right),\qquad q=p^{2}+\left(1-p\right)^{2},
\]
and for $w$ with $|w|\le m-1$ set 
\[
G_{w}=\sum^{m-|w|-1}_{j=0}q^{j},\qquad r_{w}=\sqrt{\mu(w)}.
\]
Then the matrix entries of $K_{m}\big|_{\mathcal{F}_{m}}$ in the
orthonormal basis $\left\{ e_{\phi}\right\} \cup\left\{ e_{w}\right\} $
are as follows. 
\begin{enumerate}
\item One has 
\[
\langle e_{\phi},K_{m}e_{\phi}\rangle=\sum^{m}_{n=0}q^{n}.
\]
Moreover, for every $w$ with $|w|\le m-1$, 
\[
\langle e_{w},K_{m}e_{\phi}\rangle=\left(2p-1\right)\sqrt{p\left(1-p\right)}\,r^{3}_{w}G_{w},
\]
and 
\[
\langle e_{w},K_{m}e_{w}\rangle=2p\left(1-p\right)\,r^{2}_{w}G_{w}.
\]
\item If $u\neq v$ are words with $|u|,|v|\le m-1$ and $u$ and $v$ are
incomparable, then 
\[
\langle e_{u},K_{m}e_{v}\rangle=0.
\]
\item If $u$ is a proper prefix of $v$ with $|v|\le m-1$, then 
\[
\langle e_{u},K_{m}e_{v}\rangle=\sigma\left(u,v\right)\left(2p-1\right)\frac{r^{3}_{v}}{r_{u}}G_{v},
\]
where 
\[
\sigma\left(u,v\right)=\begin{cases}
1-p, & \text{if \ensuremath{v} begins with \ensuremath{u0},}\\
-p, & \text{if \ensuremath{v} begins with \ensuremath{u2}.}
\end{cases}
\]
The case $v\prec u$ follows by self-adjointness. 
\end{enumerate}
\end{cor}

\begin{proof}
The first display is 
\[
\langle\phi,K_{m}\phi\rangle=\sum_{|u|\le m}\mu(u)^{2}=\sum^{m}_{n=0}\sum_{|w|=n}\mu(w)^{2}=\sum^{m}_{n=0}q^{n}.
\]
For $|w|\le m-1$, 
\begin{align*}
\langle e_{w},K_{m}e_{\phi}\rangle & =\frac{1}{\sqrt{p\left(1-p\right)\mu(w)}}\langle h^{\left(p\right)}_{w},K_{m}\phi\rangle,\\
\langle e_{w},K_{m}e_{w}\rangle & =\frac{1}{p\left(1-p\right)\mu(w)}\langle h^{\left(p\right)}_{w},K_{m}h^{\left(p\right)}_{w}\rangle,
\end{align*}
so the stated formulas follow by dividing \prettyref{eq:4-3} and
\prettyref{eq:4-4} by the normalization factors in \prettyref{eq:4-5}.
The vanishing in (2) is \prettyref{cor:4-4}.

Finally, if $u\prec v$, then 
\[
\langle e_{u},K_{m}e_{v}\rangle=\frac{1}{p\left(1-p\right)\sqrt{\mu(u)\mu(v)}}\langle h^{\left(p\right)}_{u},K_{m}h^{\left(p\right)}_{v}\rangle,
\]
and the stated cases follow by dividing the corresponding formula
in \prettyref{thm:4-1}(5) by the same normalization. The reverse
comparable case follows from self-adjointness. 
\end{proof}

\section{The limit operator $K_{\infty}$ for $p\protect\neq1/2$}\label{sec:5}

We now pass from the finite operators $K_{m}$ to their limit. The
filtration formula from \prettyref{prop:4-3} gives a norm convergent
series representation for $K_{\infty}$ and yields immediate control
of the approximation error $\left\Vert K_{\infty}-K_{m}\right\Vert $.
This produces a compact positive self-adjoint operator on $L^{2}\left(\mu_{p}\right)$.

The next step is to identify the internal shape of this limit operator.
The weighted Haar family from \prettyref{sec:4} extends from the
finite spaces $\mathcal{F}_{m}$ to an orthonormal basis of $L^{2}\left(\mu_{p}\right)$,
and in that basis $K_{\infty}$ has an explicit infinite matrix supported
on comparable pairs of words. The final corollary then extracts the
first finite-dimensional compression from this matrix.
\begin{prop}
\label{prop:5-1} Set 
\[
\alpha=\max\left\{ p,1-p\right\} ,\qquad q=p^{2}+\left(1-p\right)^{2}.
\]
Then the following hold. 
\begin{enumerate}
\item The series 
\[
K_{\infty}=\sum^{\infty}_{n=0}D_{n}E_{n}
\]
converges in operator norm on $L^{2}\left(\mu_{p}\right)$. Moreover,
for every $m\ge0$, 
\[
\left\Vert K_{\infty}-K_{m}\right\Vert \le\sum^{\infty}_{n=m+1}\alpha^{n}=\frac{\alpha^{m+1}}{1-\alpha}.
\]
In particular, $K_{\infty}$ is compact, positive, and self-adjoint. 
\item Writing $\lambda_{\max}\left(T\right)$ for the top eigenvalue of
a compact positive self-adjoint operator $T$, 
\[
\lambda_{\max}\left(K_{m}\right)\to\lambda_{\max}\left(K_{\infty}\right),\qquad\left|\lambda_{\max}\left(K_{\infty}\right)-\lambda_{\max}\left(K_{m}\right)\right|\le\frac{\alpha^{m+1}}{1-\alpha}.
\]
\end{enumerate}
\end{prop}

\begin{proof}
For (1), note that $\left\Vert E_{n}\right\Vert =1$ and $D_{n}$
is multiplication by the level-$n$ cylinder mass function 
\[
x\mapsto\sum_{|w|=n}\mu_{p}\left(C_{w}\right)1_{C_{w}}\left(x\right).
\]
Hence 
\[
\left\Vert D_{n}\right\Vert =\max_{|w|=n}\mu_{p}\left(C_{w}\right)=\alpha^{n}.
\]
Therefore 
\[
\sum^{\infty}_{n=0}\left\Vert D_{n}E_{n}\right\Vert \le\sum^{\infty}_{n=0}\left\Vert D_{n}\right\Vert =\sum^{\infty}_{n=0}\alpha^{n}<\infty,
\]
so $\sum^{\infty}_{n=0}D_{n}E_{n}$ converges in operator norm, and
\[
\left\Vert K_{\infty}-K_{m}\right\Vert \le\sum^{\infty}_{n=m+1}\left\Vert D_{n}E_{n}\right\Vert \le\sum^{\infty}_{n=m+1}\alpha^{n}=\frac{\alpha^{m+1}}{1-\alpha}.
\]
Since $K_{\infty}$ is a norm limit of the finite-rank operators $K_{m}$,
it is compact. Positivity and self-adjointness are preserved under
norm limits.

For (2), since $K_{m}$ and $K_{\infty}$ are compact positive self-adjoint
operators, 
\[
\left|\lambda_{\max}\left(K_{\infty}\right)-\lambda_{\max}\left(K_{m}\right)\right|\le\left\Vert K_{\infty}-K_{m}\right\Vert ,
\]
and the stated bound follows from (1).
\end{proof}

We next pass the weighted Haar formulas of \prettyref{thm:4-1} to
the limit and obtain an explicit matrix model for $K_{\infty}$ on
the full space.
\begin{prop}
\label{prop:5-2} Fix $0<p<1$ and let $K_{\infty}$ be the norm-limit
operator from \prettyref{prop:5-1}. Write 
\[
\phi=1_{C},\qquad\mu(w)=\mu_{p}\left(C_{w}\right),\qquad q=p^{2}+\left(1-p\right)^{2}.
\]
For each finite word $w\in\left\{ 0,2\right\} ^{<\infty}$, define
\[
e_{\phi}=\phi,\qquad e_{w}=\frac{h^{\left(p\right)}_{w}}{\sqrt{p\left(1-p\right)\mu(w)}}
\]
as in \prettyref{eq:4-5}. Then $\left\{ e_{\phi}\right\} \cup\left\{ e_{w}:w\in\left\{ 0,2\right\} ^{<\infty}\right\} $
is an orthonormal basis of $L^{2}\left(\mu_{p}\right)$, and the matrix
of $K_{\infty}$ in this basis is real symmetric and tree-banded:
\[
\langle e_{u},K_{\infty}e_{v}\rangle=0\qquad\text{unless \ensuremath{u} and \ensuremath{v} are comparable.}
\]

More precisely, the matrix entries are given by: 
\begin{enumerate}
\item One has 
\[
\langle e_{\phi},K_{\infty}e_{\phi}\rangle=\sum^{\infty}_{n=0}q^{n}=\frac{1}{1-q}.
\]
For every finite word $w$, 
\[
\langle e_{w},K_{\infty}e_{\phi}\rangle=\frac{2p-1}{2\sqrt{p\left(1-p\right)}}\,\mu(w)^{3/2},\qquad\langle e_{w},K_{\infty}e_{w}\rangle=\mu(w).
\]
\item If $u\neq v$ are incomparable, then $\langle e_{u},K_{\infty}e_{v}\rangle=0$.
\item If $u$ is a proper prefix of $v$, then 
\[
\langle e_{u},K_{\infty}e_{v}\rangle=\begin{cases}
{\displaystyle \frac{2p-1}{2p}\,\frac{\mu(v)^{3/2}}{\mu(u)^{1/2}},} & \text{if \ensuremath{v} begins with \ensuremath{u0},}\\[2.2ex]
{\displaystyle -\frac{2p-1}{2\left(1-p\right)}\,\frac{\mu(v)^{3/2}}{\mu(u)^{1/2}},} & \text{if \ensuremath{v} begins with \ensuremath{u2}.}
\end{cases}
\]
The case $v\prec u$ follows by self-adjointness. 
\end{enumerate}
\end{prop}

\begin{proof}
We first note that for each $m\ge1$, Theorem \prettyref{thm:4-1}
shows that $\left\{ e_{\phi}\right\} \cup\left\{ e_{w}:|w|\le m-1\right\} $
is an orthonormal basis of $\mathcal{F}_{m}$. In particular, the
family is orthonormal on the union $\bigcup_{m\ge1}\mathcal{F}_{m}$.

The identities 
\[
1_{w0}=h^{\left(p\right)}_{w}+p\,1_{w},\qquad1_{w2}=\left(1-p\right)1_{w}-h^{\left(p\right)}_{w}
\]
imply inductively that every cylinder indicator $1_{C_{u}}$ lies
in the span of 
\[
\left\{ \phi\right\} \cup\left\{ h^{\left(p\right)}_{w}:w\in\left\{ 0,2\right\} ^{<\infty}\right\} ,
\]
hence also in the span of the normalized family $\left\{ e_{\phi}\right\} \cup\left\{ e_{w}\right\} $.
Since the linear span of cylinder indicators is dense in $L^{2}\left(\mu_{p}\right)$,
the orthonormal family $\left\{ e_{\phi}\right\} \cup\left\{ e_{w}\right\} $
is complete and therefore is an orthonormal basis.

Next, fix finite words $u,v$ and note that $K_{m}\to K_{\infty}$
in operator norm by \prettyref{prop:5-1}, hence 
\[
\langle e_{u},K_{\infty}e_{v}\rangle=\lim_{m\to\infty}\langle e_{u},K_{m}e_{v}\rangle.
\]
The finite-level coefficients $\langle e_{u},K_{m}e_{v}\rangle$ are
given by \prettyref{thm:4-1} after normalization by \prettyref{eq:4-5}.
In each case the only dependence on $m$ is through the geometric
sum $\sum^{m-|w|-1}_{j=0}q^{j}$, which converges to $\sum^{\infty}_{j=0}q^{j}=1/(1-q)$.
Since $1-q=2p\left(1-p\right)$, the limits simplify to the stated
formulas.

Finally, the vanishing for incomparable indices holds already at the
finite level by \prettyref{thm:4-1}(5), hence also in the limit. 
\end{proof}

The matrix model in \prettyref{prop:5-2} yields a canonical family
of finite-dimensional compressions. For example, one may restrict
$K_{\infty}$ to the spans of weighted Haar vectors supported on a
fixed branch, or more generally to the spans generated by words up
to a given depth. These compressions form an increasing sequence,
and their top eigenvalues give monotone lower bounds for $\lambda_{\max}\left(K_{\infty}\right)$.
We give only the smallest instance, which already shows the coupling
between the constant mode and the first asymmetric mode at the root.
\begin{cor}
\label{cor:5-3}Let 
\[
\phi=1_{C},\qquad h^{\left(p\right)}_{\varnothing}=\left(1-p\right)1_{C_{0}}-p\,1_{C_{2}},\qquad e_{\varnothing}=\frac{h^{\left(p\right)}_{\varnothing}}{\sqrt{p\left(1-p\right)}}.
\]
In the orthonormal basis $\left\{ \phi,e_{\varnothing}\right\} $,
the compression of $K_{\infty}$ to $\mathrm{span}\left\{ \phi,e_{\varnothing}\right\} $
has matrix 
\[
M\left(p\right)=\begin{pmatrix}\frac{1}{1-q} & \frac{\left(2p-1\right)\sqrt{p\left(1-p\right)}}{1-q}\\[1.2ex]
\frac{\left(2p-1\right)\sqrt{p\left(1-p\right)}}{1-q} & 1
\end{pmatrix}.
\]
In particular, 
\[
\lambda_{\max}\left(K_{\infty}\right)\ge\lambda_{\max}\left(M\left(p\right)\right),
\]
where 
\[
\lambda_{\max}\left(M\left(p\right)\right)=\frac{1}{2}\left(\frac{1}{1-q}+1\right)+\frac{1}{2}\sqrt{\left(\frac{1}{1-q}-1\right)^{2}+\frac{4\left(2p-1\right)^{2}p\left(1-p\right)}{\left(1-q\right)^{2}}}.
\]
\end{cor}

\begin{proof}
We use the coefficients already computed in \prettyref{thm:4-1}.
First, 
\[
\left\langle \phi,K_{m}\phi\right\rangle =\sum_{|u|\le m}\left\langle \phi,1_{C_{u}}\right\rangle \left\langle 1_{C_{u}},\phi\right\rangle =\sum_{|u|\le m}\mu_{p}\left(C_{u}\right)^{2}=\sum^{m}_{n=0}\sum_{|w|=n}\mu_{p}\left(C_{w}\right)^{2}.
\]
The level-$n$ sum is $q^{n}$, hence 
\[
\left\langle \phi,K_{m}\phi\right\rangle =\sum^{m}_{n=0}q^{n},\qquad\left\langle \phi,K_{\infty}\phi\right\rangle =\sum^{\infty}_{n=0}q^{n}=\frac{1}{1-q}.
\]
Next, by \prettyref{thm:4-1}(3) with $w=\varnothing$ and $\mu_{p}\left(C_{\varnothing}\right)=1$,
\[
\left\langle h^{\left(p\right)}_{\varnothing},K_{m}\phi\right\rangle =p\left(1-p\right)\left(2p-1\right)\sum^{m-1}_{j=0}q^{j}.
\]
Letting $m\to\infty$ gives 
\[
\left\langle h^{\left(p\right)}_{\varnothing},K_{\infty}\phi\right\rangle =\frac{p\left(1-p\right)\left(2p-1\right)}{1-q},\qquad\left\langle \phi,K_{\infty}e_{\varnothing}\right\rangle =\frac{\left(2p-1\right)\sqrt{p\left(1-p\right)}}{1-q}.
\]
Finally, by \prettyref{thm:4-1}(4) with $w=\varnothing$, 
\[
\left\langle h^{\left(p\right)}_{\varnothing},K_{m}h^{\left(p\right)}_{\varnothing}\right\rangle =2p^{2}\left(1-p\right)^{2}\sum^{m-1}_{j=0}q^{j},
\]
so 
\[
\left\langle e_{\varnothing},K_{\infty}e_{\varnothing}\right\rangle =\frac{1}{p\left(1-p\right)}\left\langle h^{\left(p\right)}_{\varnothing},K_{\infty}h^{\left(p\right)}_{\varnothing}\right\rangle =\frac{2p\left(1-p\right)}{1-q}=1,
\]
since $1-q=2p\left(1-p\right)$. This gives the stated matrix $M\left(p\right)$.

Since $K_{\infty}$ is self-adjoint and $M\left(p\right)$ is the
matrix of its compression to $\mathrm{span}\left\{ \phi,e_{\varnothing}\right\} $,
the variational principle yields 
\[
\lambda_{\max}\left(K_{\infty}\right)\ge\lambda_{\max}\left(M\left(p\right)\right).
\]
The displayed closed form for $\lambda_{\max}\left(M\left(p\right)\right)$
is the standard formula for the top eigenvalue of a real symmetric
$2\times2$ matrix. 
\end{proof}

\section{Self-similar fixed-point identity}\label{sec:6}

We now pass from the matrix description of \prettyref{sec:5} to an
intrinsic operator identity for $K_{\infty}$. The first-level decomposition
of the Cantor set induces two branch isometries on $L^{2}\left(\mu_{p}\right)$,
and these turn the self-similarity of the measure into a self-similar
equation for the limit operator. The main result of this section shows
that $K_{\infty}$ is the unique fixed point of the resulting affine
map on $B\left(L^{2}\left(\mu_{p}\right)\right)$. From this fixed-point
description we then derive a first-level block form for $K_{\infty}$
and a norm convergent potential expansion that will be used in \prettyref{sec:7}.

This construction lies near two familiar literatures. The branch isometries
coming from the first-level Cantor splitting lead to Cuntz-type relations
and to the sort of representations that appear in the IFS and wavelet
literature; see, for example, \cite{MR467330,MR1469149,MR1743534,MR3103223}.
At the same time, the linear part of the fixed-point equation is a
normal completely positive map on $B\left(L^{2}\left(\mu_{p}\right)\right)$,
so it also falls into the general theory of completely positive maps
and their fixed points; see, for example, \cite{MR1976867,MR2805505,MR2839059,MR4550310}.
Here these structures are not imposed abstractly: they are forced
by the cylinder frame operators and the Bernoulli self-similarity.

Fix $0<p<1$ and let $\mu_{p}$ be the Bernoulli Cantor measure on
$C$. Write 
\[
\phi=1_{C},\qquad P_{\phi}f=\left\langle \phi,f\right\rangle \phi.
\]
By \prettyref{prop:5-1}, the operators $K_{m}f=\sum_{|u|\le m}\left\langle 1_{C_{u}},f\right\rangle 1_{C_{u}}$
converge in operator norm to a compact positive self-adjoint operator
\[
K_{\infty}f=\sum_{u\in\left\{ 0,2\right\} ^{<\infty}}\left\langle 1_{C_{u}},f\right\rangle 1_{C_{u}}.
\]

For the two first-level branches define operators $U_{0},U_{2}:L^{2}\left(\mu_{p}\right)\to L^{2}\left(\mu_{p}\right)$
by 
\begin{align}
\left(U_{0}f\right)\left(x\right) & =p^{-1/2}1_{C_{0}}\left(x\right)f\left(S^{-1}_{0}\left(x\right)\right),\label{eq:6-1}\\
\left(U_{2}f\right)\left(x\right) & =\left(1-p\right)^{-1/2}1_{C_{2}}\left(x\right)f\left(S^{-1}_{2}\left(x\right)\right).\label{eq:6-2}
\end{align}

We begin by isolating the branch operators and the identities they
satisfy. These identities give both the orthogonal splitting of $L^{2}\left(\mu_{p}\right)$
along the two first-level branches and the transport formulas for
cylinder indicators.
\begin{lem}
\label{lem:6-1} The operators $U_{0}$ and $U_{2}$ from \prettyref{eq:6-1}--\prettyref{eq:6-2}
are isometries with orthogonal ranges. Their adjoints are given by
\[
\left(U^{*}_{0}g\right)\left(y\right)=\sqrt{p}g\left(S_{0}\left(y\right)\right),\quad\left(U^{*}_{2}g\right)\left(y\right)=\sqrt{1-p}g\left(S_{2}\left(y\right)\right),
\]
for $\mu_{p}$-a.e. $y\in C$. Moreover, 
\begin{equation}
\left.\begin{split} & U^{*}_{0}U_{0}=U^{*}_{2}U_{2}=I\\
 & U_{0}U^{*}_{0}=M_{1_{C_{0}}},\quad U_{2}U^{*}_{2}=M_{1_{C_{2}}}\\
 & U_{0}U^{*}_{0}+U_{2}U^{*}_{2}=I.
\end{split}
\right\} \label{eq:6-3}
\end{equation}
In particular, 
\[
L^{2}\left(\mu_{p}\right)=U_{0}L^{2}\left(\mu_{p}\right)\oplus U_{2}L^{2}\left(\mu_{p}\right).
\]
Finally, for every finite word $w$, 
\begin{equation}
1_{C_{0w}}=\sqrt{p}U_{0}1_{C_{w}},\quad1_{C_{2w}}=\sqrt{1-p}U_{2}1_{C_{w}}.\label{eq:6-4}
\end{equation}
\end{lem}

\begin{proof}
For $f\in L^{2}\left(\mu_{p}\right)$, using $\mu_{p}\big|_{C_{0}}=p\mu_{p}\circ S^{-1}_{0}$,
we get from \prettyref{eq:6-1} 
\[
\left\Vert U_{0}f\right\Vert ^{2}=\int_{C_{0}}p^{-1}\left|f\left(S^{-1}_{0}\left(x\right)\right)\right|^{2}d\mu_{p}\left(x\right)=\int_{C}\left|f\left(y\right)\right|^{2}d\mu_{p}\left(y\right)=\left\Vert f\right\Vert ^{2}.
\]
Thus $U_{0}$ is an isometry. The same argument using \prettyref{eq:6-2}
shows that $U_{2}$ is an isometry. Since $U_{0}f$ is supported in
$C_{0}$ and $U_{2}g$ is supported in $C_{2}$, their ranges are
orthogonal.

For $f,g\in L^{2}\left(\mu_{p}\right)$, 
\[
\left\langle U_{0}f,g\right\rangle =\int_{C_{0}}p^{-1/2}f\left(S^{-1}_{0}\left(x\right)\right)\overline{g\left(x\right)}d\mu_{p}\left(x\right)=\int_{C}f\left(y\right)\overline{\sqrt{p}g\left(S_{0}\left(y\right)\right)}d\mu_{p}\left(y\right),
\]
so 
\[
\left(U^{*}_{0}g\right)\left(y\right)=\sqrt{p}g\left(S_{0}\left(y\right)\right).
\]
Similarly, 
\[
\left(U^{*}_{2}g\right)\left(y\right)=\sqrt{1-p}g\left(S_{2}\left(y\right)\right).
\]

Since $U_{0}$ and $U_{2}$ are isometries, we obtain $U^{*}_{0}U_{0}=U^{*}_{2}U_{2}=I$.
Also, 
\[
\left(U_{0}U^{*}_{0}g\right)\left(x\right)=1_{C_{0}}\left(x\right)g\left(x\right),\quad\left(U_{2}U^{*}_{2}g\right)\left(x\right)=1_{C_{2}}\left(x\right)g\left(x\right),
\]
which proves \prettyref{eq:6-3}. The orthogonal decomposition of
$L^{2}\left(\mu_{p}\right)$ follows immediately from \prettyref{eq:6-3}.

Finally, using \prettyref{eq:6-1}, 
\[
\left(U_{0}1_{C_{w}}\right)\left(x\right)=p^{-1/2}1_{C_{0}}\left(x\right)1_{C_{w}}\left(S^{-1}_{0}\left(x\right)\right)=p^{-1/2}1_{C_{0w}}\left(x\right),
\]
which gives the first identity in \prettyref{eq:6-4}. The second
is proved in the same way using \prettyref{eq:6-2}. 
\end{proof}

\begin{rem}
\label{rem:6-2} From \prettyref{eq:6-3} one has the Cuntz relations
\[
U^{*}_{i}U_{j}=\delta_{ij}I,\qquad U_{0}U^{*}_{0}+U_{2}U^{*}_{2}=I,
\]
so $\left(U_{0},U_{2}\right)$ is a Cuntz family on $L^{2}\left(\mu_{p}\right)$.
This is natural because the Cantor set splits as $C=C_{0}\dot{\cup}C_{2}$,
and the maps $S_{0},S_{2}$ identify each branch $(C_{0},\mu_{p}\big|_{C_{0}})$
and $(C_{2},\mu_{p}\big|_{C_{2}})$ with a rescaled copy of $\left(C,\mu_{p}\right)$.
Thus $U_{0}$ and $U_{2}$ are the normalized pullbacks implementing
this two-branch self-similarity at the level of $L^{2}$. We will
use only the resulting orthogonal branch decomposition and the conjugation
identities in what follows. 
\end{rem}

With the branch identities in hand, we now define the affine map determined
by the first-level splitting and show that it governs both the finite
operators $K_{m}$ and the limit operator $K_{\infty}$.
\begin{thm}
\label{thm:6-3} Define 
\[
\Psi\left(T\right)=P_{\phi}+pU_{0}TU^{*}_{0}+\left(1-p\right)U_{2}TU^{*}_{2}
\]
for bounded operators $T$ on $L^{2}\left(\mu_{p}\right)$. Then the
following hold. 
\begin{enumerate}
\item For every $m\ge0$, 
\[
K_{m+1}=\Psi\left(K_{m}\right).
\]
In particular, $K_{m}=\Psi^{m}\left(P_{\phi}\right)$. 
\item The operator $K_{\infty}$ satisfies 
\[
K_{\infty}=\Psi\left(K_{\infty}\right).
\]
Equivalently, 
\[
K_{\infty}=P_{\phi}+pU_{0}K_{\infty}U^{*}_{0}+\left(1-p\right)U_{2}K_{\infty}U^{*}_{2}.
\]
\item $K_{\infty}$ is the unique bounded operator on $L^{2}\left(\mu_{p}\right)$
satisfying $T=\Psi\left(T\right)$. 
\end{enumerate}
\end{thm}

\begin{proof}
We prove (1). For $f\in L^{2}\left(\mu_{p}\right)$, 
\[
K_{m+1}f=\left\langle \phi,f\right\rangle \phi+\sum_{|w|\le m}\left\langle 1_{C_{0w}},f\right\rangle 1_{C_{0w}}+\sum_{|w|\le m}\left\langle 1_{C_{2w}},f\right\rangle 1_{C_{2w}}.
\]
Using \prettyref{eq:6-4}, 
\[
\sum_{|w|\le m}\left\langle 1_{C_{0w}},f\right\rangle 1_{C_{0w}}=pU_{0}\left(\sum_{|w|\le m}\left\langle 1_{C_{w}},U^{*}_{0}f\right\rangle 1_{C_{w}}\right)=pU_{0}K_{m}U^{*}_{0}f,
\]
and similarly 
\[
\sum_{|w|\le m}\left\langle 1_{C_{2w}},f\right\rangle 1_{C_{2w}}=\left(1-p\right)U_{2}K_{m}U^{*}_{2}f.
\]
Therefore $K_{m+1}=\Psi\left(K_{m}\right)$. Since $K_{0}=P_{\phi}$,
iteration gives $K_{m}=\Psi^{m}\left(P_{\phi}\right)$.

We prove (2). From the definition of $\Psi$ and $U^{*}_{0}U_{0}=U^{*}_{2}U_{2}=I$
in \prettyref{eq:6-3}, 
\[
\left\Vert \Psi\left(T\right)-\Psi\left(S\right)\right\Vert \le p\left\Vert U_{0}\left(T-S\right)U^{*}_{0}\right\Vert +\left(1-p\right)\left\Vert U_{2}\left(T-S\right)U^{*}_{2}\right\Vert \le\left\Vert T-S\right\Vert ,
\]
so $\Psi$ is norm-continuous. Taking limits in $K_{m+1}=\Psi\left(K_{m}\right)$
and using $K_{m}\to K_{\infty}$ in operator norm gives $K_{\infty}=\Psi\left(K_{\infty}\right)$.

We prove (3). For bounded operators $T,S$, 
\[
\Psi\left(T\right)-\Psi\left(S\right)=pU_{0}\left(T-S\right)U^{*}_{0}+\left(1-p\right)U_{2}\left(T-S\right)U^{*}_{2}.
\]
By \prettyref{eq:6-3}, the two summands act on the orthogonal subspaces
$U_{0}L^{2}\left(\mu_{p}\right)$ and $U_{2}L^{2}\left(\mu_{p}\right)$,
hence 
\[
\left\Vert \Psi\left(T\right)-\Psi\left(S\right)\right\Vert =\max\left\{ p\left\Vert U_{0}\left(T-S\right)U^{*}_{0}\right\Vert ,\left(1-p\right)\left\Vert U_{2}\left(T-S\right)U^{*}_{2}\right\Vert \right\} .
\]
Using again $U^{*}_{0}U_{0}=U^{*}_{2}U_{2}=I$ from \prettyref{eq:6-3},
\[
\left\Vert U_{0}\left(T-S\right)U^{*}_{0}\right\Vert =\left\Vert T-S\right\Vert =\left\Vert U_{2}\left(T-S\right)U^{*}_{2}\right\Vert ,
\]
so 
\[
\left\Vert \Psi\left(T\right)-\Psi\left(S\right)\right\Vert =\max\left\{ p,1-p\right\} \left\Vert T-S\right\Vert .
\]
Set $\alpha=\max\left\{ p,1-p\right\} <1$. Then $\Psi$ is a strict
contraction on $B\left(L^{2}\left(\mu_{p}\right)\right)$, hence has
a unique fixed point by the Banach fixed-point theorem. Since $K_{\infty}$
is a fixed point by (2), it is the unique one. 
\end{proof}

The fixed-point identity becomes more concrete after passing to the
orthogonal decomposition from \prettyref{eq:6-3}. In that decomposition,
$K_{\infty}$ takes an explicit $2\times2$ block form.
\begin{cor}
\label{cor:6-4} Let 
\[
W:L^{2}\left(\mu_{p}\right)\oplus L^{2}\left(\mu_{p}\right)\to L^{2}\left(\mu_{p}\right),\qquad W\left(f,g\right)=U_{0}f+U_{2}g.
\]
Then $W$ is unitary, and 
\begin{equation}
W^{*}K_{\infty}W=\begin{pmatrix}pK_{\infty}+pP_{\phi} & \sqrt{p\left(1-p\right)}P_{\phi}\\[1.2ex]
\sqrt{p\left(1-p\right)}P_{\phi} & \left(1-p\right)K_{\infty}+\left(1-p\right)P_{\phi}
\end{pmatrix}.\label{eq:6-5-1}
\end{equation}
\end{cor}

\begin{proof}
By \prettyref{eq:6-3}, 
\[
L^{2}\left(\mu_{p}\right)=U_{0}L^{2}\left(\mu_{p}\right)\oplus U_{2}L^{2}\left(\mu_{p}\right),
\]
so $W$ is unitary. Set 
\[
b=\left(\sqrt{p}\phi,\sqrt{1-p}\phi\right).
\]
Then $Wb=\phi$, hence 
\[
W^{*}P_{\phi}W=P_{b}.
\]
Using \prettyref{thm:6-3}, 
\[
K_{\infty}=P_{\phi}+pU_{0}K_{\infty}U^{*}_{0}+\left(1-p\right)U_{2}K_{\infty}U^{*}_{2}.
\]
Conjugating by $W$ yields 
\[
W^{*}K_{\infty}W=P_{b}+\begin{pmatrix}pK_{\infty} & 0\\
0 & \left(1-p\right)K_{\infty}
\end{pmatrix}.
\]
Since 
\[
P_{b}=\begin{pmatrix}pP_{\phi} & \sqrt{p\left(1-p\right)}P_{\phi}\\[1.2ex]
\sqrt{p\left(1-p\right)}P_{\phi} & \left(1-p\right)P_{\phi}
\end{pmatrix},
\]
the result follows. 
\end{proof}

We next separate the linear part of the affine fixed-point equation.
Since that linear part is a strict contraction, the fixed-point identity
unfolds into a norm convergent series expansion for $K_{\infty}$.
\begin{cor}
\label{cor:6-6} Define 
\[
\Phi\left(T\right)=pU_{0}TU^{*}_{0}+\left(1-p\right)U_{2}TU^{*}_{2}
\]
for $T\in B\left(L^{2}\left(\mu_{p}\right)\right)$. Then $\Phi$
is a normal completely positive map, 
\begin{equation}
\left\Vert \Phi\right\Vert =\max\left\{ p,1-p\right\} <1,\label{eq:6-6}
\end{equation}
and 
\begin{equation}
K_{\infty}=\sum^{\infty}_{n=0}\Phi^{n}\left(P_{\phi}\right)\label{eq:6-7}
\end{equation}
with convergence in operator norm. Equivalently, 
\begin{equation}
K_{\infty}=\left(I-\Phi\right)^{-1}\left(P_{\phi}\right).\label{eq:6-8}
\end{equation}
\end{cor}

\begin{proof}
Set $V_{0}=\sqrt{p}U_{0}$, $V_{2}=\sqrt{1-p}U_{2}$. Then 
\[
\Phi\left(T\right)=V_{0}TV^{*}_{0}+V_{2}TV^{*}_{2},
\]
so $\Phi$ is normal and completely positive.

By \prettyref{eq:6-3}, the ranges of $U_{0}$ and $U_{2}$ are orthogonal
and 
\[
L^{2}\left(\mu_{p}\right)=U_{0}L^{2}\left(\mu_{p}\right)\oplus U_{2}L^{2}\left(\mu_{p}\right).
\]
Hence 
\[
\left\Vert \Phi\left(T\right)\right\Vert =\max\left\{ p\left\Vert U_{0}TU^{*}_{0}\right\Vert ,\left(1-p\right)\left\Vert U_{2}TU^{*}_{2}\right\Vert \right\} =\max\left\{ p,1-p\right\} \left\Vert T\right\Vert ,
\]
which gives \prettyref{eq:6-6}.

By \prettyref{thm:6-3}, $K_{\infty}=P_{\phi}+\Phi\left(K_{\infty}\right)$,
$K_{0}=P_{\phi}$, and $K_{m+1}=P_{\phi}+\Phi\left(K_{m}\right)$.
It follows by induction that 
\[
K_{m}=\sum^{m}_{n=0}\Phi^{n}\left(P_{\phi}\right).
\]
Since $\left\Vert \Phi\right\Vert <1$, the series on the right converges
in operator norm. Passing to the limit as $m\to\infty$ and using
$K_{m}\to K_{\infty}$ yields \prettyref{eq:6-7}. Finally, \prettyref{eq:6-8}
is the Neumann series form of \prettyref{eq:6-7}. 
\end{proof}

\section{Resolvent of $K_{\infty}$}\label{sec:7}

We now turn from the operator identities of \prettyref{sec:6} to
scalar information extracted from the resolvent of $K_{\infty}$.
The block form from \prettyref{cor:6-4} turns the first-level self-similarity
into a nonlinear identity for the scalar resolvent function $m\left(z\right)$.
From this renormalization formula we then derive a recursion for the
rooted moments, the associated spectral measure at $\phi$, and a
scalar characterization of the top eigenvalue.

We begin by converting the block decomposition of $K_{\infty}$ into
an identity for the scalar resolvent at the root vector $\phi$.
\begin{thm}
\label{thm:7-1} Fix $0<p<1$ and let $K_{\infty}$ be the compact
positive self-adjoint operator on $L^{2}\left(\mu_{p}\right)$ constructed
in \prettyref{prop:5-1}. Write $\phi=1_{C}$ and define 
\[
m\left(z\right)=\left\langle \phi,\left(zI-K_{\infty}\right)^{-1}\phi\right\rangle .
\]
Then, for every $z\in\mathbb{C}$ with 
\[
\left|z\right|>\left\Vert K_{\infty}\right\Vert +1,
\]
one has 
\begin{equation}
m\left(z\right)=\frac{m\left(z/p\right)}{1-m\left(z/p\right)}+\frac{m\left(z/\left(1-p\right)\right)}{\left(1-m\left(z/p\right)\right)\left(1-m\left(z/p\right)-m\left(z/\left(1-p\right)\right)\right)}.\label{eq:7-1}
\end{equation}
\end{thm}

\begin{proof}
Let 
\[
W:L^{2}\left(\mu_{p}\right)\oplus L^{2}\left(\mu_{p}\right)\to L^{2}\left(\mu_{p}\right),\qquad W\left(f,g\right)=U_{0}f+U_{2}g,
\]
and set 
\[
b=\left(\sqrt{p}\phi,\sqrt{1-p}\phi\right).
\]
By \prettyref{cor:6-4}, $W$ is unitary, 
\[
W^{*}K_{\infty}W=\begin{pmatrix}pK_{\infty}+pP_{\phi} & \sqrt{p\left(1-p\right)}P_{\phi}\\[1.2ex]
\sqrt{p\left(1-p\right)}P_{\phi} & \left(1-p\right)K_{\infty}+\left(1-p\right)P_{\phi}
\end{pmatrix},
\]
and $Wb=\phi$. Hence 
\[
m\left(z\right)=\left\langle b,\left(zI-A\right)^{-1}b\right\rangle ,
\]
where 
\[
A:=\begin{pmatrix}pK_{\infty}+pP_{\phi} & \sqrt{p\left(1-p\right)}P_{\phi}\\[1.2ex]
\sqrt{p\left(1-p\right)}P_{\phi} & \left(1-p\right)K_{\infty}+\left(1-p\right)P_{\phi}
\end{pmatrix}.
\]

Write 
\[
c=\sqrt{p\left(1-p\right)},\qquad P=P_{\phi},
\]
and 
\[
zI-A=\begin{pmatrix}A_{11} & B\\
B & A_{22}
\end{pmatrix},
\]
where 
\begin{equation}
A_{11}=zI-pK_{\infty}-pP,\qquad A_{22}=zI-\left(1-p\right)K_{\infty}-\left(1-p\right)P,\qquad B=-cP.\label{eq:7-2}
\end{equation}

Assume now that $\left|z\right|>\left\Vert K_{\infty}\right\Vert +1$.
Since $A=W^{*}K_{\infty}W$, we have $\left\Vert A\right\Vert =\left\Vert K_{\infty}\right\Vert $,
so $zI-A$ is invertible. Also, 
\[
\left\Vert pK_{\infty}+pP\right\Vert \le p\left\Vert K_{\infty}\right\Vert +p<\left|z\right|,
\]
and similarly 
\[
\left\Vert \left(1-p\right)K_{\infty}+\left(1-p\right)P\right\Vert \le\left(1-p\right)\left\Vert K_{\infty}\right\Vert +\left(1-p\right)<\left|z\right|,
\]
so both $A_{11}$ and $A_{22}$ are invertible. Hence the Schur complement
\[
S=A_{22}-BA^{-1}_{11}B
\]
is invertible.

Next set 
\begin{equation}
R_{p}\left(z\right)=\left(zI-pK_{\infty}\right)^{-1},\qquad R_{1-p}\left(z\right)=\left(zI-\left(1-p\right)K_{\infty}\right)^{-1}.\label{eq:7-3}
\end{equation}
These are well defined because 
\[
\left\Vert pK_{\infty}\right\Vert \le p\left\Vert K_{\infty}\right\Vert <\left|z\right|,\qquad\left\Vert \left(1-p\right)K_{\infty}\right\Vert \le\left(1-p\right)\left\Vert K_{\infty}\right\Vert <\left|z\right|.
\]
Moreover, 
\begin{equation}
R_{p}\left(z\right)=\frac{1}{p}\left(\frac{z}{p}I-K_{\infty}\right)^{-1},\qquad R_{1-p}\left(z\right)=\frac{1}{1-p}\left(\frac{z}{1-p}I-K_{\infty}\right)^{-1},\label{eq:7-4}
\end{equation}
so 
\begin{equation}
\left\langle \phi,R_{p}\left(z\right)\phi\right\rangle =\frac{1}{p}m\left(z/p\right),\qquad\left\langle \phi,R_{1-p}\left(z\right)\phi\right\rangle =\frac{1}{1-p}m\left(z/\left(1-p\right)\right).\label{eq:7-5}
\end{equation}

Apply the rank-one resolvent identity to 
\[
A_{11}=zI-pK_{\infty}-pP.
\]
This gives 
\[
A^{-1}_{11}=R_{p}\left(z\right)+pR_{p}\left(z\right)\phi\frac{1}{1-p\left\langle \phi,R_{p}\left(z\right)\phi\right\rangle }\left\langle \phi,R_{p}\left(z\right)\cdot\right\rangle .
\]
Hence 
\[
g_{p}\left(z\right):=\left\langle \phi,A^{-1}_{11}\phi\right\rangle =\frac{\left\langle \phi,R_{p}\left(z\right)\phi\right\rangle }{1-p\left\langle \phi,R_{p}\left(z\right)\phi\right\rangle }.
\]
Using \prettyref{eq:7-5}, 
\begin{equation}
g_{p}\left(z\right)=\frac{\frac{1}{p}m\left(z/p\right)}{1-m\left(z/p\right)},\qquad pg_{p}\left(z\right)=\frac{m\left(z/p\right)}{1-m\left(z/p\right)}.\label{eq:7-6}
\end{equation}

Since $PA^{-1}_{11}P=g_{p}\left(z\right)P$, \prettyref{eq:7-2} gives
\[
S=A_{22}-BA^{-1}_{11}B=zI-\left(1-p\right)K_{\infty}-\beta\left(z\right)P,
\]
where 
\[
\beta\left(z\right)=\left(1-p\right)+p\left(1-p\right)g_{p}\left(z\right)=\left(1-p\right)\left(1+pg_{p}\left(z\right)\right).
\]
Applying the rank-one resolvent identity again, 
\[
s\left(z\right):=\left\langle \phi,S^{-1}\phi\right\rangle =\frac{\left\langle \phi,R_{1-p}\left(z\right)\phi\right\rangle }{1-\beta\left(z\right)\left\langle \phi,R_{1-p}\left(z\right)\phi\right\rangle }.
\]
Using \prettyref{eq:7-5} and \prettyref{eq:7-6}, 
\begin{equation}
\left(1-p\right)s\left(z\right)=\frac{m\left(z/\left(1-p\right)\right)}{1-\left(1+pg_{p}\left(z\right)\right)m\left(z/\left(1-p\right)\right)}.\label{eq:7-7}
\end{equation}

We now evaluate $\left\langle b,\left(zI-A\right)^{-1}b\right\rangle $
using the block inverse formula. The $(2,2)$-block of $\left(zI-A\right)^{-1}$
is $S^{-1}$, and the $(1,2)$-block is 
\[
-A^{-1}_{11}BS^{-1}=cA^{-1}_{11}PS^{-1}.
\]
A direct computation gives 
\begin{equation}
m\left(z\right)=pg_{p}\left(z\right)+\left(1-p\right)s\left(z\right)\left(1+pg_{p}\left(z\right)\right)^{2}.\label{eq:7-8}
\end{equation}
Now \prettyref{eq:7-6} gives 
\[
pg_{p}\left(z\right)=\frac{m\left(z/p\right)}{1-m\left(z/p\right)},\qquad1+pg_{p}\left(z\right)=\frac{1}{1-m\left(z/p\right)}.
\]
Substituting these and \prettyref{eq:7-7} into \prettyref{eq:7-8},
we obtain 
\[
m\left(z\right)=\frac{m\left(z/p\right)}{1-m\left(z/p\right)}+\frac{m\left(z/\left(1-p\right)\right)}{\left(1-m\left(z/p\right)\right)^{2}}\frac{1}{1-\frac{m\left(z/\left(1-p\right)\right)}{1-m\left(z/p\right)}}.
\]
This simplifies to \prettyref{eq:7-1}. 
\end{proof}

Once the resolvent satisfies a renormalization identity, its expansion
at infinity yields recursive information on the moments of the rooted
spectral measure.
\begin{prop}
\label{prop:7-2} For $n\ge0$, set 
\[
\mu_{n}=\left\langle \phi,K^{n}_{\infty}\phi\right\rangle .
\]
Then, for $\left|z\right|>\left\Vert K_{\infty}\right\Vert $, 
\[
m\left(z\right)=\sum^{\infty}_{n=0}\frac{\mu_{n}}{z^{n+1}}.
\]
Moreover, \prettyref{thm:7-1} determines the sequence $\left(\mu_{n}\right)_{n\ge0}$
recursively: for each $n\ge1$ there is a polynomial 
\[
P_{n}\in\mathbb{Q}\left[p\right]\left[X_{0},\dots,X_{n-1}\right]
\]
such that 
\[
\left(1-p^{n+1}-\left(1-p\right)^{n+1}\right)\mu_{n}=P_{n}\left(\mu_{0},\dots,\mu_{n-1}\right).
\]
In particular, $\mu_{n}$ is uniquely determined by $\mu_{0},\dots,\mu_{n-1}$,
since 
\[
1-p^{n+1}-\left(1-p\right)^{n+1}>0
\]
for every $n\ge1$.

The first moments are 
\begin{align*}
\mu_{0} & =1,\\
\mu_{1} & =\frac{1}{2p\left(1-p\right)}=\frac{1}{1-q},\qquad q=p^{2}+\left(1-p\right)^{2},\\
\mu_{2} & =\frac{p^{2}-p+1}{3p^{2}\left(1-p\right)^{2}},\\
\mu_{3} & =\frac{12p^{4}-24p^{3}+38p^{2}-26p+11}{24p^{3}\left(1-p\right)^{3}\left(p^{2}-p+2\right)}.
\end{align*}
\end{prop}

\begin{proof}
Since $\left|z\right|>\left\Vert K_{\infty}\right\Vert $, the Neumann
series gives 
\[
\left(zI-K_{\infty}\right)^{-1}=\frac{1}{z}\sum^{\infty}_{n=0}\frac{K^{n}_{\infty}}{z^{n}}
\]
with convergence in operator norm. Pairing against $\phi$ yields
\[
m\left(z\right)=\left\langle \phi,\left(zI-K_{\infty}\right)^{-1}\phi\right\rangle =\sum^{\infty}_{n=0}\frac{\mu_{n}}{z^{n+1}}.
\]

Set 
\[
a\left(z\right)=m\left(z/p\right),\qquad b\left(z\right)=m\left(z/\left(1-p\right)\right).
\]
Then 
\[
a\left(z\right)=\sum^{\infty}_{n=0}\mu_{n}p^{n+1}z^{-\left(n+1\right)},\qquad b\left(z\right)=\sum^{\infty}_{n=0}\mu_{n}\left(1-p\right)^{n+1}z^{-\left(n+1\right)}.
\]
By \prettyref{thm:7-1}, 
\[
m\left(z\right)=\frac{a\left(z\right)}{1-a\left(z\right)}+\frac{b\left(z\right)}{\left(1-a\left(z\right)\right)\left(1-a\left(z\right)-b\left(z\right)\right)}.
\]
Since $a\left(z\right),b\left(z\right)=O\left(z^{-1}\right)$ as $\left|z\right|\to\infty$,
each denominator admits a convergent geometric expansion in powers
of $z^{-1}$. Therefore the coefficient of $z^{-\left(n+1\right)}$
on the right-hand side is a polynomial expression in $\mu_{0},\dots,\mu_{n}$
with coefficients in $\mathbb{Q}\left[p\right]$.

The only contribution involving $\mu_{n}$ linearly comes from the
first-order terms in the geometric expansions, namely from 
\[
a\left(z\right)+b\left(z\right).
\]
Hence the coefficient of $z^{-\left(n+1\right)}$ on the right-hand
side is 
\[
\mu_{n}p^{n+1}+\mu_{n}\left(1-p\right)^{n+1}+Q_{n}\left(\mu_{0},\dots,\mu_{n-1}\right),
\]
where 
\[
Q_{n}\in\mathbb{Q}\left[p\right]\left[X_{0},\dots,X_{n-1}\right].
\]
Comparing with the coefficient $\mu_{n}$ on the left gives 
\[
\mu_{n}=\mu_{n}p^{n+1}+\mu_{n}\left(1-p\right)^{n+1}+Q_{n}\left(\mu_{0},\dots,\mu_{n-1}\right),
\]
that is, 
\[
\left(1-p^{n+1}-\left(1-p\right)^{n+1}\right)\mu_{n}=Q_{n}\left(\mu_{0},\dots,\mu_{n-1}\right).
\]
This is the stated recursion, after renaming $Q_{n}$ as $P_{n}$.
Since $0<p<1$, one has 
\[
0<p^{n+1}+\left(1-p\right)^{n+1}<1
\]
for every $n\ge1$, so the coefficient of $\mu_{n}$ is strictly positive.

We now compute the first terms. Write 
\[
m\left(z\right)=\frac{1}{z}+\frac{\mu_{1}}{z^{2}}+\frac{\mu_{2}}{z^{3}}+\frac{\mu_{3}}{z^{4}}+O\left(z^{-5}\right).
\]
Then 
\[
m\left(z/p\right)=\frac{p}{z}+\frac{\mu_{1}p^{2}}{z^{2}}+\frac{\mu_{2}p^{3}}{z^{3}}+\frac{\mu_{3}p^{4}}{z^{4}}+O\left(z^{-5}\right),
\]
and 
\[
m\left(z/\left(1-p\right)\right)=\frac{1-p}{z}+\frac{\mu_{1}\left(1-p\right)^{2}}{z^{2}}+\frac{\mu_{2}\left(1-p\right)^{3}}{z^{3}}+\frac{\mu_{3}\left(1-p\right)^{4}}{z^{4}}+O\left(z^{-5}\right).
\]
Substituting into the renormalization identity and comparing coefficients
of $z^{-2}$, $z^{-3}$, and $z^{-4}$ gives 
\[
\mu_{1}=\frac{1}{2p\left(1-p\right)},
\]
\[
\mu_{2}=\frac{p^{2}-p+1}{3p^{2}\left(1-p\right)^{2}},
\]
and 
\[
\mu_{3}=\frac{12p^{4}-24p^{3}+38p^{2}-26p+11}{24p^{3}\left(1-p\right)^{3}\left(p^{2}-p+2\right)}.
\]
Finally, the identity 
\[
\mu_{1}=\frac{1}{1-q}
\]
follows from $1-q=2p\left(1-p\right)$, in agreement with \prettyref{prop:5-1}. 
\end{proof}

The first few terms give immediate checks against formulas already
obtained earlier and against the symmetric case.
\begin{cor}
\label{cor:7-3} With the notation of \prettyref{prop:7-2}, the following
hold. 
\begin{enumerate}
\item The first moment agrees with \prettyref{prop:5-1}: 
\[
\mu_{1}=\left\langle \phi,K_{\infty}\phi\right\rangle =\frac{1}{1-q},\qquad q=p^{2}+\left(1-p\right)^{2}.
\]
\item In the symmetric case $p=\frac{1}{2}$, 
\[
\mu_{1}=2,\qquad\mu_{2}=4,\qquad\mu_{3}=8.
\]
\end{enumerate}
\end{cor}

\begin{proof}
For $(1)$, \prettyref{prop:7-2} gives 
\[
\mu_{1}=\frac{1}{2p\left(1-p\right)}.
\]
Since 
\[
1-q=1-p^{2}-\left(1-p\right)^{2}=2p\left(1-p\right),
\]
it follows that 
\[
\mu_{1}=\frac{1}{1-q}.
\]
This is exactly the identity already obtained in \prettyref{prop:5-1}.

For $(2)$, substitute $p=\frac{1}{2}$ into the formulas of \prettyref{prop:7-2}.
Then 
\begin{align*}
\mu_{1} & =\frac{1}{2\cdot\frac{1}{2}\cdot\frac{1}{2}}=2,\qquad\mu_{2}=\frac{\frac{1}{4}-\frac{1}{2}+1}{3\cdot\frac{1}{4}\cdot\frac{1}{4}}=4,\\
\mu_{3} & =\frac{12\cdot\frac{1}{16}-24\cdot\frac{1}{8}+38\cdot\frac{1}{4}-26\cdot\frac{1}{2}+11}{24\cdot\frac{1}{8}\cdot\frac{1}{8}\cdot\left(\frac{1}{4}-\frac{1}{2}+2\right)}=8.
\end{align*}
\end{proof}

The moment sequence comes from a distinguished spectral measure obtained
by evaluating the spectral theorem at the root vector $\phi$.
\begin{cor}
\label{cor:7-4} There is a unique finite positive Borel measure $\nu_{\phi}$
on $\left[0,\left\Vert K_{\infty}\right\Vert \right]$ such that 
\[
m\left(z\right)=\left\langle \phi,\left(zI-K_{\infty}\right)^{-1}\phi\right\rangle =\int_{\left[0,\left\Vert K_{\infty}\right\Vert \right]}\frac{1}{z-\lambda}d\nu_{\phi}\left(\lambda\right)
\]
for every $z\in\mathbb{C}\setminus\left[0,\left\Vert K_{\infty}\right\Vert \right]$.
Moreover, 
\[
\mu_{n}=\left\langle \phi,K^{n}_{\infty}\phi\right\rangle =\int_{\left[0,\left\Vert K_{\infty}\right\Vert \right]}\lambda^{n}d\nu_{\phi}\left(\lambda\right)\qquad\left(n\ge0\right),
\]
so \prettyref{prop:7-2} determines $\nu_{\phi}$ uniquely. 
\end{cor}

\begin{proof}
Since $K_{\infty}$ is compact, positive, and self-adjoint by \prettyref{prop:5-1},
the spectral theorem gives a unique finite positive Borel measure
$\nu_{\phi}$ supported on the spectrum of $K_{\infty}$ such that
\[
\left\langle \phi,f\left(K_{\infty}\right)\phi\right\rangle =\int f\left(\lambda\right)d\nu_{\phi}\left(\lambda\right)
\]
for every bounded Borel function $f$ on $\sigma\left(K_{\infty}\right)$.
Taking $f\left(\lambda\right)=\frac{1}{z-\lambda}$ gives the stated
representation of $m\left(z\right)$, and taking $f\left(\lambda\right)=\lambda^{n}$
gives $\mu_{n}=\int\lambda^{n}\,d\nu_{\phi}\left(\lambda\right)$. 

Since $\nu_{\phi}$ is supported on the compact interval $\left[0,\left\Vert K_{\infty}\right\Vert \right]$,
the Hausdorff moment problem on a compact interval is determinate.
Hence the moment sequence $\left(\mu_{n}\right)_{n\ge0}$ determines
$\nu_{\phi}$ uniquely. By \prettyref{prop:7-2}, these moments are
recursively determined. 
\end{proof}

In the symmetric case, the rooted measure collapses to a single atom
because $\phi$ is already an eigenvector.
\begin{cor}
\label{cor:7-5} Assume $p=\frac{1}{2}$. Then $K_{\infty}\phi=2\phi$.
Consequently, 
\[
\mu_{n}=\left\langle \phi,K^{n}_{\infty}\phi\right\rangle =2^{n}\qquad\left(n\ge0\right),
\]
and the rooted spectral measure from \prettyref{cor:7-4} is 
\[
\nu_{\phi}=\delta_{2}.
\]
In particular, the values 
\[
\mu_{1}=2,\qquad\mu_{2}=4,\qquad\mu_{3}=8
\]
from \prettyref{cor:7-3} agree with the direct eigenvector computation
above. 
\end{cor}

\begin{proof}
Assume $p=\frac{1}{2}$. Then \prettyref{prop:4-3} reduces to $K_{m}=\sum^{m}_{n=0}2^{-n}E_{n}$.
Passing to the norm limit and using \prettyref{prop:5-1}, we get
$K_{\infty}=\sum^{\infty}_{n=0}2^{-n}E_{n}$ with convergence in operator
norm. Since $E_{n}\phi=\phi$ for every $n$, it follows that $K_{\infty}\phi=\sum^{\infty}_{n=0}2^{-n}\phi=2\phi$.
Therefore $K^{n}_{\infty}\phi=2^{n}\phi$ for every $n\ge0$, and
hence 
\[
\mu_{n}=\left\langle \phi,K^{n}_{\infty}\phi\right\rangle =2^{n}\left\langle \phi,\phi\right\rangle =2^{n}.
\]

Now let $\nu_{\phi}$ be the measure from \prettyref{cor:7-4}. Since
$\phi$ is an eigenvector of $K_{\infty}$ with eigenvalue $2$, the
spectral measure at $\phi$ is the point mass at $2$, so $\nu_{\phi}=\delta_{2}$.
The identities $\mu_{1}=2$, $\mu_{2}=4$, $\mu_{3}=8$ now follow
again, and agree with the values already obtained in \prettyref{cor:7-3}
by substituting $p=\frac{1}{2}$ into the formulas of \prettyref{prop:7-2}. 
\end{proof}

We close the section by using the block decomposition once more, now
to reduce the top-eigenvalue problem to a scalar equation involving
the rooted resolvent function.
\begin{thm}
\label{thm:7-6} Let 
\[
L=\lambda_{\max}\left(K_{\infty}\right),\qquad\alpha=\max\left\{ p,1-p\right\} ,
\]
and let 
\[
m\left(z\right)=\left\langle \phi,\left(zI-K_{\infty}\right)^{-1}\phi\right\rangle ,\qquad z\in\mathbb{C}\setminus\sigma\left(K_{\infty}\right).
\]
Then for every real number $\lambda>\alpha L$, the following are
equivalent:
\begin{enumerate}
\item $\lambda$ is an eigenvalue of $K_{\infty}$. 
\item One has 
\[
m\left(\lambda/p\right)+m\left(\lambda/\left(1-p\right)\right)=1.
\]
\end{enumerate}
Moreover, every eigenvalue $\lambda>\alpha L$ is simple. In particular,
$L$ is the largest real number $\lambda>\alpha L$ satisfying 
\[
m\left(\lambda/p\right)+m\left(\lambda/\left(1-p\right)\right)=1,
\]
and the top eigenvalue $L$ is simple.

\end{thm}

\begin{proof}
This is the standard rank-one perturbation (Birman-Schwinger) criterion,
specialized to the decomposition $A=D+P_{b}$ below; see, for example,
\cite{MR1335452,MR4369236,MR2154153}. We include the argument for
completeness.

Let $A=W^{*}K_{\infty}W$ be the block form from \prettyref{cor:6-4}.
Thus 
\[
A=\begin{pmatrix}pK_{\infty}+pP_{\phi} & \sqrt{p\left(1-p\right)}\,P_{\phi}\\[1.2ex]
\sqrt{p\left(1-p\right)}\,P_{\phi} & \left(1-p\right)K_{\infty}+\left(1-p\right)P_{\phi}
\end{pmatrix}.
\]
Write 
\[
D=\begin{pmatrix}pK_{\infty} & 0\\
0 & \left(1-p\right)K_{\infty}
\end{pmatrix},\qquad b=\left(\sqrt{p}\,\phi,\sqrt{1-p}\,\phi\right).
\]
Then 
\[
A=D+P_{b}.
\]

Since $K_{\infty}$ is compact, positive, and self-adjoint by \prettyref{prop:5-1},
its spectrum is contained in $\left[0,L\right]$. Therefore 
\[
\sigma\left(D\right)=p\,\sigma\left(K_{\infty}\right)\cup\left(1-p\right)\sigma\left(K_{\infty}\right)\subset\left[0,\alpha L\right].
\]
Hence, if $\lambda>\alpha L$, then $\lambda\notin\sigma\left(D\right)$,
so $\lambda I-D$ is invertible.

We first prove that $(1)$ implies $(2)$. Since $A$ is unitarily
equivalent to $K_{\infty}$, the number $\lambda$ is an eigenvalue
of $K_{\infty}$ if and only if it is an eigenvalue of $A$. So assume
\[
Ax=\lambda x
\]
for some nonzero vector $x\in L^{2}\left(\mu_{p}\right)\oplus L^{2}\left(\mu_{p}\right)$.
Then 
\[
\left(\lambda I-D\right)x=P_{b}x=\left\langle b,x\right\rangle b.
\]
We claim that $\left\langle b,x\right\rangle \neq0$. Indeed, if $\left\langle b,x\right\rangle =0$,
then $\left(\lambda I-D\right)x=0$, and since $\lambda I-D$ is invertible,
this would force $x=0$, a contradiction. Thus 
\[
x=\left\langle b,x\right\rangle \left(\lambda I-D\right)^{-1}b.
\]
Taking the inner product with $b$ and cancelling the nonzero scalar
$\left\langle b,x\right\rangle $, we obtain 
\[
1=\left\langle b,\left(\lambda I-D\right)^{-1}b\right\rangle .
\]
Now 
\[
\left(\lambda I-D\right)^{-1}=\begin{pmatrix}\left(\lambda I-pK_{\infty}\right)^{-1} & 0\\
0 & \left(\lambda I-\left(1-p\right)K_{\infty}\right)^{-1}
\end{pmatrix},
\]
so 
\[
\left\langle b,\left(\lambda I-D\right)^{-1}b\right\rangle =p\left\langle \phi,\left(\lambda I-pK_{\infty}\right)^{-1}\phi\right\rangle +\left(1-p\right)\left\langle \phi,\left(\lambda I-\left(1-p\right)K_{\infty}\right)^{-1}\phi\right\rangle .
\]
Using 
\begin{align*}
\left(\lambda I-pK_{\infty}\right)^{-1} & =\frac{1}{p}\left(\frac{\lambda}{p}I-K_{\infty}\right)^{-1},\\
\left(\lambda I-\left(1-p\right)K_{\infty}\right)^{-1} & =\frac{1}{1-p}\left(\frac{\lambda}{1-p}I-K_{\infty}\right)^{-1},
\end{align*}
we get 
\[
\left\langle b,\left(\lambda I-D\right)^{-1}b\right\rangle =m\left(\lambda/p\right)+m\left(\lambda/\left(1-p\right)\right).
\]
Therefore 
\[
m\left(\lambda/p\right)+m\left(\lambda/\left(1-p\right)\right)=1.
\]

We now prove that $(2)$ implies $(1)$. Assume 
\[
m\left(\lambda/p\right)+m\left(\lambda/\left(1-p\right)\right)=1.
\]
By the computation above, this is equivalent to 
\[
\left\langle b,\left(\lambda I-D\right)^{-1}b\right\rangle =1.
\]
Set 
\[
x=\left(\lambda I-D\right)^{-1}b.
\]
Then $x\neq0$, since $b\neq0$ and $\lambda I-D$ is invertible.
Moreover, 
\[
\left(\lambda I-A\right)x=\left(\lambda I-D-P_{b}\right)x=b-\left\langle b,x\right\rangle b.
\]
Because 
\[
\left\langle b,x\right\rangle =\left\langle b,\left(\lambda I-D\right)^{-1}b\right\rangle =1,
\]
it follows that 
\[
\left(\lambda I-A\right)x=0.
\]
Hence $\lambda$ is an eigenvalue of $A$, and therefore of $K_{\infty}$.

We next prove simplicity. Let $\lambda>\alpha L$ be an eigenvalue.
If 
\[
Ay=\lambda y,
\]
then the argument above gives 
\[
y=\left\langle b,y\right\rangle \left(\lambda I-D\right)^{-1}b.
\]
Thus every eigenvector for $\lambda$ is a scalar multiple of $\left(\lambda I-D\right)^{-1}b$.
So the eigenspace is one-dimensional, and $\lambda$ is simple.

Since $0<p<1$, one has $\alpha<1$. Also $L>0$, since $K_{\infty}\neq0$.
Hence $L>\alpha L$, so the equivalence already proved applies at
$\lambda=L$ and gives 
\[
m\left(L/p\right)+m\left(L/\left(1-p\right)\right)=1.
\]

If $\lambda>\alpha L$ is any other real solution, then by the equivalence
proved above, $\lambda$ is an eigenvalue of $K_{\infty}$. Since
$L=\lambda_{\max}\left(K_{\infty}\right)$ is the largest eigenvalue
of $K_{\infty}$, it follows that $\lambda\le L$. Therefore $L$
is the largest real number $\lambda>\alpha L$ satisfying 
\[
m\left(\lambda/p\right)+m\left(\lambda/\left(1-p\right)\right)=1.
\]
Its simplicity follows from the simplicity statement already proved. 
\end{proof}

\begin{rem}
Several parts of the construction do not depend on the specific geometry.
The filtration formula $K_{m}=\sum^{m}_{n=0}D_{n}E_{n}$ is a measure-theoretic
identity attached to the cylinder partition filtration, and it extends
to any probability space equipped with a finite refining partition
tree.

Likewise, the self-similar fixed-point identity for $K_{\infty}$
extends to any self-similar measure arising from a finite family of
injective contractions whose first-level pieces are disjoint modulo
the measure. In that setting the branch maps define isometries with
orthogonal ranges, and the same fixed-point argument applies. We restrict
attention here to the Bernoulli Cantor case in order to keep the weighted
and symmetric structures explicit. 
\end{rem}

\bibliographystyle{amsalpha}
\bibliography{ref}

@article{MR3990189,
	author = {Anoussis, M. and Katavolos, A. and Todorov, I. G.},
	doi = {10.4064/sm180313-6-9},
	fjournal = {Studia Mathematica},
	issn = {0039-3223,1730-6337},
	journal = {Studia Math.},
	mrclass = {43A20 (22D15 22D25 43A77 47L05)},
	mrnumber = {3990189},
	mrreviewer = {Jos\'e\ Extremera},
	number = {2},
	pages = {193--213},
	title = {Bimodules over {${\rm VN}(G)$}, harmonic operators and the non-commutative {P}oisson boundary},
	url = {https://doi.org/10.4064/sm180313-6-9},
	volume = {249},
	year = {2019}}

@article{MR4126800,
	author = {Hiai, Fumio},
	doi = {10.1007/s43036-019-00034-9},
	fjournal = {Advances in Operator Theory},
	issn = {2662-2009,2538-225X},
	journal = {Adv. Oper. Theory},
	mrclass = {47A64 (47B65 47L07 58B20)},
	mrnumber = {4126800},
	number = {3},
	pages = {680--713},
	title = {Operator means deformed by a fixed point method},
	url = {https://doi.org/10.1007/s43036-019-00034-9},
	volume = {5},
	year = {2020}}

@article{MR4478259,
	author = {Gabe, James},
	doi = {10.4171/jncg/479},
	fjournal = {Journal of Noncommutative Geometry},
	issn = {1661-6952,1661-6960},
	journal = {J. Noncommut. Geom.},
	mrclass = {46L05 (19K35 46L07 46L35 46L80)},
	mrnumber = {4478259},
	mrreviewer = {Anil\ Kumar\ Karn},
	number = {2},
	pages = {391--421},
	title = {Lifting theorems for completely positive maps},
	url = {https://doi.org/10.4171/jncg/479},
	volume = {16},
	year = {2022}}

@article{MR4718688,
	author = {Li, Yuan and Gao, Shuhui and Zhao, Cong and Ma, Nan},
	doi = {10.1007/s11117-024-01037-4},
	fjournal = {Positivity. An International Mathematics Journal Devoted to Theory and Applications of Positivity},
	issn = {1385-1292,1572-9281},
	journal = {Positivity},
	mrclass = {47L05 (47A05 47B10)},
	mrnumber = {4718688},
	mrreviewer = {Janko\ Bra\v ci\v c},
	number = {2},
	pages = {Paper No. 22, 17},
	title = {On spectra of some completely positive maps},
	url = {https://doi.org/10.1007/s11117-024-01037-4},
	volume = {28},
	year = {2024}}

@article{MR2572704,
	author = {Weihua, Liu and Junde, Wu},
	doi = {10.1063/1.3253574},
	fjournal = {Journal of Mathematical Physics},
	issn = {0022-2488,1089-7658},
	journal = {J. Math. Phys.},
	mrclass = {47H10 (81R15)},
	mrnumber = {2572704},
	number = {10},
	pages = {103531, 2},
	title = {On fixed points of {L}\"uders operation},
	url = {https://doi.org/10.1063/1.3253574},
	volume = {50},
	year = {2009}}

@article{MR3059439,
	author = {Szehr, Oleg and Wolf, Michael M.},
	doi = {10.1063/1.4795112},
	fjournal = {Journal of Mathematical Physics},
	issn = {0022-2488,1089-7658},
	journal = {J. Math. Phys.},
	mrclass = {81S25 (60J20)},
	mrnumber = {3059439},
	number = {3},
	pages = {032203, 10},
	title = {Perturbation bounds for quantum {M}arkov processes and their fixed points},
	url = {https://doi.org/10.1063/1.4795112},
	volume = {54},
	year = {2013}}

@article{MR4670347,
	author = {Mograby, Gamal and Balu, Radhakrishnan and Okoudjou, Kasso A. and Teplyaev, Alexander},
	doi = {10.4171/jst/473},
	fjournal = {Journal of Spectral Theory},
	issn = {1664-039X,1664-0403},
	journal = {J. Spectr. Theory},
	mrclass = {47B36 (05E05 28A80 47A10 81Q35 81T17)},
	mrnumber = {4670347},
	mrreviewer = {Netanel\ Levi},
	number = {3},
	pages = {903--935},
	title = {Spectral decimation of piecewise centrosymmetric {J}acobi operators on graphs},
	url = {https://doi.org/10.4171/jst/473},
	volume = {13},
	year = {2023}}

@incollection{MR2441420,
	author = {Kumagai, Takashi},
	booktitle = {Selected papers on analysis and related topics},
	doi = {10.1090/trans2/223/06},
	isbn = {978-0-8218-3928-7},
	mrclass = {28A80 (31C25 35K05 46E35 60J75)},
	mrnumber = {2441420},
	pages = {81--95},
	publisher = {Amer. Math. Soc., Providence, RI},
	series = {Amer. Math. Soc. Transl. Ser. 2},
	title = {Recent developments of analysis on fractals [MR2105988]},
	url = {https://doi.org/10.1090/trans2/223/06},
	volume = {223},
	year = {2008}}

@article{MR2254554,
	author = {Hambly, B. M. and Metz, V. and Teplyaev, A.},
	doi = {10.1112/S002461070602312X},
	fjournal = {Journal of the London Mathematical Society. Second Series},
	issn = {0024-6107,1469-7750},
	journal = {J. London Math. Soc. (2)},
	mrclass = {31C25 (28A80 47H07 60J65)},
	mrnumber = {2254554},
	mrreviewer = {Uta\ R.\ Freiberg},
	number = {1},
	pages = {93--112},
	title = {Self-similar energies on post-critically finite self-similar fractals},
	url = {https://doi.org/10.1112/S002461070602312X},
	volume = {74},
	year = {2006}}

@incollection{MR1410793,
	author = {Kigami, Jun},
	fjournal = {S\=urikaisekikenky\=usho K\B oky\=uroku},
	journal = {S\=urikaisekikenky\=usho K\B oky\=uroku},
	mrclass = {28A80 (31C25 60J15 92C15)},
	mrnumber = {1410793},
	note = {Low-dimensional dynamical systems and related topics (Japanese) (Kyoto, 1995)},
	number = {938},
	pages = {133--148},
	title = {Laplacians on self-similar sets and their spectral distributions},
	year = {1996}}

@book{MR1840042,
	author = {Kigami, Jun},
	doi = {10.1017/CBO9780511470943},
	isbn = {0-521-79321-1},
	mrclass = {28A80 (31C20 31C25 35J05 35K05)},
	mrnumber = {1840042},
	mrreviewer = {Volker\ Metz},
	pages = {viii+226},
	publisher = {Cambridge University Press, Cambridge},
	series = {Cambridge Tracts in Mathematics},
	title = {Analysis on fractals},
	url = {https://doi.org/10.1017/CBO9780511470943},
	volume = {143},
	year = {2001}}

@article{MR1998572,
	author = {Kribs, David W.},
	doi = {10.1017/S0013091501000980},
	fjournal = {Proceedings of the Edinburgh Mathematical Society. Series II},
	issn = {0013-0915,1464-3839},
	journal = {Proc. Edinb. Math. Soc. (2)},
	mrclass = {42C40 (46L60 47A20 81P68 94A12)},
	mrnumber = {1998572},
	number = {2},
	pages = {421--433},
	title = {Quantum channels, wavelets, dilations and representations of {$O_n$}},
	url = {https://doi.org/10.1017/S0013091501000980},
	volume = {46},
	year = {2003}}

@article{MR3343663,
	author = {Farkov, Yu. A.},
	doi = {10.1134/S000143461411039X},
	fjournal = {Mathematical Notes},
	issn = {0001-4346,1573-8876},
	journal = {Math. Notes},
	mrclass = {43A60 (42C40)},
	mrnumber = {3343663},
	mrreviewer = {Victor\ Alexandrovich\ Makarichev},
	note = {Translation of Mat. Zametki {\bf 96} (2014), no. 6, 926--938},
	number = {5-6},
	pages = {996--1007},
	title = {Wavelet expansions on the {C}antor group},
	url = {https://doi.org/10.1134/S000143461411039X},
	volume = {96},
	year = {2014}}

@article{MR1785282,
	author = {Strichartz, Robert S.},
	doi = {10.1007/BF02788990},
	fjournal = {Journal d'Analyse Math\'ematique},
	issn = {0021-7670,1565-8538},
	journal = {J. Anal. Math.},
	mrclass = {42A38 (28A75 42C05)},
	mrnumber = {1785282},
	mrreviewer = {Steen\ Pedersen},
	pages = {209--238},
	title = {Mock {F}ourier series and transforms associated with certain {C}antor measures},
	url = {https://doi.org/10.1007/BF02788990},
	volume = {81},
	year = {2000}}

@article{MR1648441,
	author = {Jorgensen, Palle E. T. and Pedersen, Steen},
	doi = {10.1016/S0764-4442(97)82984-9},
	fjournal = {Comptes Rendus de l'Acad\'emie des Sciences. S\'erie I. Math\'ematique},
	issn = {0764-4442},
	journal = {C. R. Acad. Sci. Paris S\'er. I Math.},
	mrclass = {28A80 (42B10 46E30)},
	mrnumber = {1648441},
	mrreviewer = {Paul\ Jolissaint},
	number = {3},
	pages = {301--306},
	title = {Orthogonal harmonic analysis and scaling of fractal measures},
	url = {https://doi.org/10.1016/S0764-4442(97)82984-9},
	volume = {326},
	year = {1998}}

@article{MR4369236,
	author = {Behrndt, Jussi and ter Elst, A. F. M. and Gesztesy, Fritz},
	doi = {10.1090/tran/8401},
	fjournal = {Transactions of the American Mathematical Society},
	issn = {0002-9947,1088-6850},
	journal = {Trans. Amer. Math. Soc.},
	mrclass = {47A53 (47A55)},
	mrnumber = {4369236},
	mrreviewer = {Hans-Olav\ Tylli},
	number = {2},
	pages = {799--845},
	title = {The generalized {B}irman-{S}chwinger principle},
	url = {https://doi.org/10.1090/tran/8401},
	volume = {375},
	year = {2022}}

@book{MR1335452,
	author = {Kato, Tosio},
	isbn = {3-540-58661-X},
	mrclass = {47A55 (46-00 47-00)},
	mrnumber = {1335452},
	note = {Reprint of the 1980 edition},
	pages = {xxii+619},
	publisher = {Springer-Verlag, Berlin},
	series = {Classics in Mathematics},
	title = {Perturbation theory for linear operators},
	year = {1995}}

@article{MR4550310,
	author = {Li, Yuan and Li, Fan and Chen, Shan and Chen, Yanni},
	doi = {10.1016/S0034-4877(23)00014-9},
	fjournal = {Reports on Mathematical Physics},
	issn = {0034-4877,1879-0674},
	journal = {Rep. Math. Phys.},
	mrclass = {47B65 (47B10 47C15 47N50 81P47)},
	mrnumber = {4550310},
	mrreviewer = {Harsh\ Trivedi},
	number = {1},
	pages = {117--129},
	title = {Approximation states and fixed points of quantum channels},
	url = {https://doi.org/10.1016/S0034-4877(23)00014-9},
	volume = {91},
	year = {2023}}

@article{MR3103223,
	author = {Dutkay, Dorin Ervin and Picioroaga, Gabriel and Song, Myung-Sin},
	doi = {10.1016/j.jmaa.2013.07.012},
	fjournal = {Journal of Mathematical Analysis and Applications},
	issn = {0022-247X,1096-0813},
	journal = {J. Math. Anal. Appl.},
	mrclass = {42C10 (46L05)},
	mrnumber = {3103223},
	mrreviewer = {J\'er\^ome\ Gilles},
	number = {2},
	pages = {1128--1139},
	title = {Orthonormal bases generated by {C}untz algebras},
	url = {https://doi.org/10.1016/j.jmaa.2013.07.012},
	volume = {409},
	year = {2014}}

@article{MR2839059,
	author = {Li, Yuan},
	doi = {10.1063/1.3583541},
	fjournal = {Journal of Mathematical Physics},
	issn = {0022-2488,1089-7658},
	journal = {J. Math. Phys.},
	mrclass = {81R15 (47H10 81P15)},
	mrnumber = {2839059},
	mrreviewer = {Pavel\ V.\ Exner},
	number = {5},
	pages = {052103, 9},
	title = {Characterizations of fixed points of quantum operations},
	url = {https://doi.org/10.1063/1.3583541},
	volume = {52},
	year = {2011}}

@article{MR2805505,
	author = {Li, Yuan},
	doi = {10.1016/j.jmaa.2011.04.047},
	fjournal = {Journal of Mathematical Analysis and Applications},
	issn = {0022-247X,1096-0813},
	journal = {J. Math. Anal. Appl.},
	mrclass = {47N50 (47B15 47B65 81R15 81S22)},
	mrnumber = {2805505},
	mrreviewer = {Julio\ Guerrero},
	number = {1},
	pages = {172--179},
	title = {Fixed points of dual quantum operations},
	url = {https://doi.org/10.1016/j.jmaa.2011.04.047},
	volume = {382},
	year = {2011}}

@book{MR1976867,
	author = {Paulsen, Vern},
	isbn = {0-521-81669-6},
	mrclass = {46L07 (47A20 47L30)},
	mrnumber = {1976867},
	mrreviewer = {Christian\ Le Merdy},
	pages = {xii+300},
	publisher = {Cambridge University Press, Cambridge},
	series = {Cambridge Studies in Advanced Mathematics},
	title = {Completely bounded maps and operator algebras},
	volume = {78},
	year = {2002}}

@article{MR1743534,
	author = {Bratteli, Ola and Evans, David E. and Jorgensen, Palle E. T.},
	doi = {10.1006/acha.2000.0283},
	fjournal = {Applied and Computational Harmonic Analysis. Time-Frequency and Time-Scale Analysis, Wavelets, Numerical Algorithms, and Applications},
	issn = {1063-5203,1096-603X},
	journal = {Appl. Comput. Harmon. Anal.},
	mrclass = {46L52 (42C40 46L05)},
	mrnumber = {1743534},
	number = {2},
	pages = {166--196},
	title = {Compactly supported wavelets and representations of the {C}untz relations},
	url = {https://doi.org/10.1006/acha.2000.0283},
	volume = {8},
	year = {2000}}

@article{MR1469149,
	author = {Bratteli, Ola and Jorgensen, Palle E. T.},
	doi = {10.1090/memo/0663},
	fjournal = {Memoirs of the American Mathematical Society},
	issn = {0065-9266,1947-6221},
	journal = {Mem. Amer. Math. Soc.},
	mrclass = {46Lxx (11B85 22D25 39B12 42C15 47D25)},
	mrnumber = {1469149},
	mrreviewer = {Paul\ Jolissaint},
	number = {663},
	pages = {x+89},
	title = {Iterated function systems and permutation representations of the {C}untz algebra},
	url = {https://doi.org/10.1090/memo/0663},
	volume = {139},
	year = {1999}}

@article{MR467330,
	author = {Cuntz, Joachim},
	fjournal = {Communications in Mathematical Physics},
	issn = {0010-3616,1432-0916},
	journal = {Comm. Math. Phys.},
	mrclass = {46L05},
	mrnumber = {467330},
	mrreviewer = {E.\ St\o rmer},
	number = {2},
	pages = {173--185},
	title = {Simple {$C\sp*$}-algebras generated by isometries},
	url = {http://projecteuclid.org/euclid.cmp/1103901288},
	volume = {57},
	year = {1977}}

@book{MR1228209,
	author = {Meyer, Yves},
	isbn = {0-521-42000-8; 0-521-45869-2},
	mrclass = {42-02 (42C15 46E30 47G99)},
	mrnumber = {1228209},
	note = {Translated from the 1990 French original by D. H. Salinger},
	pages = {xvi+224},
	publisher = {Cambridge University Press, Cambridge},
	series = {Cambridge Studies in Advanced Mathematics},
	title = {Wavelets and operators},
	volume = {37},
	year = {1992}}

@book{MR2154153,
	author = {Simon, Barry},
	doi = {10.1090/surv/120},
	edition = {Second},
	isbn = {0-8218-3581-5},
	mrclass = {47L20 (47A40 47A55 47B10 47B36 47E05 81Q15 81U99)},
	mrnumber = {2154153},
	mrreviewer = {Pavel\ B.\ Kurasov},
	pages = {viii+150},
	publisher = {American Mathematical Society, Providence, RI},
	series = {Mathematical Surveys and Monographs},
	title = {Trace ideals and their applications},
	url = {https://doi.org/10.1090/surv/120},
	volume = {120},
	year = {2005}}

@incollection{MR864712,
	author = {Burkholder, Donald L.},
	booktitle = {Probability and analysis ({V}arenna, 1985)},
	doi = {10.1007/BFb0076300},
	isbn = {3-540-16787-0},
	mrclass = {42A45 (42A61 46E40 60G42 60G46)},
	mrnumber = {864712},
	mrreviewer = {Terry\ R.\ McConnell},
	pages = {61--108},
	publisher = {Springer, Berlin},
	series = {Lecture Notes in Math.},
	title = {Martingales and {F}ourier analysis in {B}anach spaces},
	url = {https://doi.org/10.1007/BFb0076300},
	volume = {1206},
	year = {1986}}

@book{MR618663,
	author = {Edwards, R. E. and Gaudry, G. I.},
	isbn = {3-540-07726-X},
	mrclass = {42A18 (43A22)},
	mrnumber = {618663},
	pages = {ix+212},
	publisher = {Springer-Verlag, Berlin-New York},
	series = {Ergebnisse der Mathematik und ihrer Grenzgebiete [Results in Mathematics and Related Areas]},
	title = {Littlewood-{P}aley and multiplier theory},
	volume = {Band 90},
	year = {1977}}

@book{MR402915,
	author = {Neveu, J.},
	edition = {Revised},
	mrclass = {60G45 (60G40)},
	mrnumber = {402915},
	note = {Translated from the French by T. P. Speed},
	pages = {viii+236},
	publisher = {North-Holland Publishing Co., Amsterdam-Oxford; American Elsevier Publishing Co., Inc., New York},
	series = {North-Holland Mathematical Library},
	title = {Discrete-parameter martingales},
	volume = {Vol. 10},
	year = {1975}}

@book{MR833073,
	author = {Kahane, Jean-Pierre},
	edition = {Second},
	isbn = {0-521-24966-X; 0-521-45602-9},
	mrclass = {60G99 (42A61 60G15 60G17)},
	mrnumber = {833073},
	mrreviewer = {Michael\ Marcus},
	pages = {xiv+305},
	publisher = {Cambridge University Press, Cambridge},
	series = {Cambridge Studies in Advanced Mathematics},
	title = {Some random series of functions},
	volume = {5},
	year = {1985}}

\end{document}